\documentclass[13pt, reqno]{amsart}
\usepackage{amsaddr}
\usepackage{latexsym,amsfonts}
\usepackage{amssymb,amsthm,upref,amscd}
\usepackage[T1]{fontenc}
\usepackage{times}
\usepackage{amscd}
\usepackage{enumerate}
\usepackage{mathrsfs}
\usepackage{amsmath}
\usepackage{color}
\usepackage{soul}
\usepackage{indentfirst}
\usepackage{comment}
\PassOptionsToPackage{reqno}{amsmath}

\usepackage{amsmath,amsfonts,amsthm,amssymb,bbm}
\usepackage{graphicx,color,dsfont}
\usepackage{enumitem}
\usepackage{fourier}

\usepackage{mathtools}
\mathtoolsset{showonlyrefs}

\newtheorem{thm}{Theorem}[section]
\newtheorem{prop}{Proposition}[section]
\newtheorem{defi}{Definition}[section]
\newtheorem{lem}{Lemma}[section]

\newtheorem{rem}{Remark}[section]

\newcommand{\R}{\mathbb{R}}

\numberwithin{equation}{section}
\newcommand{\N}{\mathbb{N}}

\newcommand{\eps}{\epsilon}

\newcommand{\wto}{\rightharpoonup}

\makeatletter \@addtoreset{equation}{section} \makeatother

\usepackage[textwidth=138mm,
textheight=215mm,
left=30mm,
right=30mm,
top=25.4mm,
bottom=25.4mm,
headheight=2.17cm,
headsep=4mm,
footskip=12mm,
heightrounded]{geometry}

\usepackage[colorlinks=true,urlcolor=blue,
citecolor=black,linkcolor=black,linktocpage,pdfpagelabels,
bookmarksnumbered,bookmarksopen]{hyperref}
\usepackage{hyperref}
\newcounter{const}
\author[T. Gou]{Tianxiang Gou}
\address[T. Gou]{
\centerline{School of Mathematics and Statistics, Xi’an Jiaotong University,}
\centerline{710049, Xi’an, Shaanxi, China}}



\email{tianxiang.gou@xjtu.edu.cn}
\title[Non-degeneracy and Uniqueness]{Non-degeneracy and uniqueness of ground states to nonlinear elliptic equations with mixed local and nonlocal operators}
\thanks{{\it Acknowledgments.} T. Gou was supported by the National Natural Science Foundation of China (No. 12471113).}
\thanks{{\it Conflict of interest statement}. The author declares that there is no conflict of interest.}

\thanks{{\it Data availability statement}. The author affirms that our paper has no associated data.}

\begin{document}

\begin{abstract} 

This paper concerns the non-degeneracy and uniqueness of ground states to the following nonlinear elliptic equation with mixed local and nonlocal operators,
$$
-\Delta u +(-\Delta)^s u + \lambda u=|u|^{p-2}u \quad \mbox{in} \,\,\, B, \quad u=0 \quad \mbox{in} \,\,\, \R^N \backslash {B},
$$
where $N \geq 2$, $0<s<1$, $2<p<2^*:=\frac{2N}{(N-2)^+}$, $\lambda > -\lambda_1$, $(-\Delta)^s$ denotes the fractional Laplacian, $\lambda_1>0$ denotes the first Dirichlet eigenvalue of the operator $-\Delta +(-\Delta)^s$ in $B$ and $B$ denotes the unit ball in $\R^N$. We prove that the second eigenvalue to the linearized operator $-\Delta +(-\Delta)^s -(p-1)u^{p-2}$ in the space of radially symmetric functions is simple, the corresponding eigenfunction changes sign precisely once in the radial direction, where $u$ is a ground state. By deriving a new Hopf type lemma, we then get that $-\lambda$ cannot be an eigenvalue of the linearized operator, which in turns leads to the non-degeneracy of ground states. Moreover, by establishing a Picone type identity with respect to antisymmetric functions, we then derive the non-degeneracy of ground states in the space of non-radially symmetric functions. Relying on the non-degeneracy of ground states and adapting a blow-up argument together with a continuation argument, we then obtain the uniqueness of ground states.
\end{abstract}

\maketitle

\thispagestyle{empty}

\section{Introduction}

In this paper, we investigate the non-degeneracy and uniqueness of ground states to the following nonlinear elliptic equation with mixed local and nonlocal operators,
\begin{align} \label{equ}
-\Delta u +(-\Delta)^s u + \lambda u=|u|^{p-2}u \quad \mbox{in} \,\,\,B, \quad u=0 \quad \mbox{in} \,\,\,\R^N \backslash {B},
\end{align}
where $N \geq 2$, $0<s<1$, $2<p<2^*:=\frac{2N}{(N-2)^+}$, $\lambda > -\lambda_1$, $\lambda_1>0$ denotes the first Dirichlet eigenvalue of the operator $-\Delta +(-\Delta)^s$ in $B$, $B$ denotes the unit ball in $\R^N$ defined by $B:=\{x \in \R^N : |x|<1\}$ and $(-\Delta)^s$ denotes the fractional Laplacian defined by
$$
(-\Delta)^s u(x):=C(N, s) P.V. \int_{\R^N} \frac{u(x)-u(y)}{|x-y|^{n+2s}} \,dy, \quad C(N, s):=\frac{2^{2s}{s} \Gamma(\frac{N+2s}{2})}{\pi^{\frac N 2}\Gamma(1-s)}.
$$

The operator $\mathcal{L}:=-\Delta +(-\Delta)^s$ arises naturally from the superposition of two stochastic processes with different scales (i.e. Brownian motions and L\'evy flights). Loosely speaking, when a particle follows either of these two processes according to a certain probability, then the associated limit diffusion equation can be described by the operator $\mathcal{L}$, see for example \cite{DLV, DV}. The operator generally models the combined effect of local and nonlocal diffusions in many situations, which enjoys a constantly rising popularity in applications and appears in bi-modal power-law distribution processes, the theory of optimal searching strategies, biomathematics and animal foraging, see for example \cite{Bd, DLV, DV, PV} and references therein.

Elliptic equations driven by the mixed local and nonlocal operators as previously have recently received a great attention from several points of view, including existence, symmetry, maximum principles, interior regularity and boundary regularity, see for example \cite{BCCI, BI, BDVV1, BDVV2, BDVV, BDVV3, BDVV4, BM, BMS, BJ, CDV, CKSV, DLV, JK, CM, SVWZ1, SVWZ2, SVWZ}. However, non-degeneracy and uniqueness have not been well-understood so far. The aim of the present paper is to focus on these two topics as a complement of the study carried out previously. More precisely, we shall explore the non-degeneracy and uniqueness of ground states to \eqref{equ}. 

The study of the non-degeneracy and uniqueness of solutions, which has a long history, is particularly interesting but challenging. For the local case $s=1$, taking advantage of ODE techniques, the uniqueness of positive solutions to the following equation has been discussed in the seminal work of Kwong \cite{K} (see also \cite{C, K, M, MS, NN, Y, Z}),
\begin{align} \label{equ0}
(-\Delta)^s u + \lambda u=|u|^{p-2}u \quad \mbox{in} \,\,\, \Omega, \quad u=0 \quad \mbox{in} \,\,\,  \R^N \backslash \Omega,
\end{align}
where $\Omega=B$ or $\R^N$. For the nonlocal case $0<s<1$, the study of the uniqueness and non-degeneracy of solutions to \eqref{equ0} becomes more challenging and there exist few results up to now, because the classical ODE techniques are not applicable. In this case, when $\Omega=\R$, making use of complex analytic methods, Amick and Toland in \cite{AT} initially established the uniqueness of solutions to \eqref{equ0} with $s=1/2$ and $p=3$.
The result was recently extended by Frank et al. \cite{FL, FLS} to more general cases, see also \cite{FV, FW1}. While $\Omega=B_R$, the consideration of the uniqueness and non-degeneracy of solutions to \eqref{equ0}  becomes different and tough, because one needs to deal with boundary terms generated from applying fractional integration by parts formula given in \cite{RS} and new arguments are required to deal with the terms, where $B_R:=\{x \in \R^N : |x|<R\}$ for $R>0$. To our knowledge, there exist only few results obtained recently in this direction, we refer the readers to \cite{DP, DIS, DIS1, FW} for the pertinent research.

In contrast with problems only having Laplacian operator or fractional Laplacian operator, the one under our consideration with the mixed operators combines features from the local and nonlocal operators, which does not enjoy scale-invariant any longer and either one of the operators prevail or both coexist. This then brings out new difficulties and interesting phenomena.


Let us now address main results of the present paper. For the sake of completeness, we shall first sketch the existence of ground states to \eqref{equ}. Let $H^1_0(B)$ denote the usual Sobolev space defined by the completion of $C^{\infty}_0(B)$ under the standard norm
$$
\|u\|_{H_0^1}:=\left(\|u\|_2^2 + \|\nabla u\|_2^2\right)^{\frac 12},
$$ 
where
$$
\|u\|_p:=\left(\int_{B} |u|^p \,dx \right)^{\frac 1p}, \quad 1 \leq p <+\infty, \quad \|u\|_{\infty}:=\mathop{\mbox{ess sup}}\limits_{x \in B} |f(x)|.
$$
Let $H^1_{0,r}(B)$ be the subspace consisting of radially symmetric functions in $H^1_0(B)$. We now introduce the natural energy space asscoiated to \eqref{equ}, which is indeed defined by
$$
X_0^1(B):=\left\{u : \R^N \to \R \,\, \mbox{is Lebesgue measurable} : u \mid_{B} \in H^1_0(B), [u]_s<+\infty, u=0 \,\,\mbox{a.e. in}\,\, \R^N \backslash {B}\right\}
$$
equipped with the norm
$$
\|u\|_{X_0^1}:=\left(\|u\|_{H^1_0}^2 + [u]_s^2 \right)^{\frac 12},
$$
where
\begin{align} \label{defn}
[u]_s:=\left(\frac{C(N,s)}{2}\int_{\R^N} \int_{\R^N} \frac{|u(x)-u(y)|^2}{|x-y|^{N+2s}}\,dxdy\right)^{\frac 12}=\left(\int_{\R^N} |(-\Delta)^{\frac s2} u|^2 \, dx\right)^{\frac 12}.
\end{align}
As an application of \cite[Lemma 2.1]{SVWZ}, one knows that $X_0^1(B)$ is indeed equivalent to $H^1_0(B)$ in the sense that $\|\cdot\|_{X_0^1} \sim \|\cdot\|_{H^1_0}$. Consequently, we are going to explore ground states to \eqref{equ} in $H^1_0(B)$.

\begin{defi}
\begin{itemize}
\item[$(\textnormal{i})$] We say that $u \in H^1_0(B) \backslash \{0\}$ is a (weak) solution to \eqref{equ} if it holds that, for any $\psi \in C^{\infty}_0(B)$,
$$
\int_{B} \nabla u \cdot \nabla \psi \,dx +\frac{C(N,s)}{2} \int_{\R^N} \int_{\R^N} \frac{\left(u(x)-u(y)\right)\left(\psi(x)-\psi(y)\right)}{|x-y|^{N+2s}}\,dxdy + \lambda \int_{B} u \psi \,dx = \int_{B} |u|^{p-2} u \psi \,dx.
$$
\end{itemize}
\item[$(\textnormal{ii})$]  We say that $u \in H^1_0(B) \backslash \{0\}$ is a ground state to \eqref{equ} if it is a solution to \eqref{equ} and possesses the least energy among all solutions, where the underlying energy functional associated to \eqref{equ} is defined by
$$
I(u):=\frac 12 \int_{B} |\nabla u|^2 \,dx + \frac{C(N,s)}{4}  \int_{\R^N} \int_{\R^N} \frac{|u(x)-u(y)|^2}{|x-y|^{N+2s}}\,dxdy +\frac {\lambda}{2} \int_{B} |u|^2 \,dx-\frac 1p \int_{B} |u|^p \,dx.
$$

\end{defi}

To establish the existence of ground states to \eqref{equ}, we shall employ variational arguments. It is straightforward to check that $I$ is of class $C^1$ in $H^1_0(B)$. Let us now introduce the following minimization problem,
\begin{align} \label{min}
m:=\inf_{u \in N} I(u),
\end{align}
where$$
N:=\left\{ u \in H^1_0(B) \backslash \{0\} : \langle I'(u), u \rangle =0\right\}.
$$
To start with, applying the well-known Sobolev's inequality in $H^1_0(B)$, we are able to conclude that $m>0$. Since $H_0^1(B)$ is compactly embedded into $L^p(B)$ for any $2 \leq p<2^*$, then we know that any minimizing sequence to \eqref{min} is compact in $H^1_0(B)$. As a consequence, the existence of minimizers to \eqref{min} follows immediately. Observe that if $u \in H^1_0(B)$ is a minimizer to \eqref{min}, then $|u| \in H^1_0(B)$ is also a minimizer to \eqref{min}. Therefore, without restriction, we may assume that any minimizer to \eqref{min} is non-negative. Since Nehari's manifold $N$ is a natural constraint, then any minimizer to \eqref{min} is a ground state to \eqref{equ}. On the other hand, one finds that any ground state to \eqref{equ} is a minimizer to \eqref{min}. Thereby, we conclude that \eqref{equ} possesses ground states and any ground state to \eqref{equ} is non-negative. 

Let $u \in H^1_0(B)$ be a ground state to \eqref{equ}. It follows from \cite[Theorem 1.1]{SVWZ} that $u \in L^{\infty}(B)$. In view of \cite[Theorem 1.4]{SVWZ}, we then know that $u \in C_{loc}^{1,\alpha}(B)$ for any $0<\alpha<1$. Further, thanks to \cite[Theorem 1.5]{SVWZ}, we have that $u \in C_{loc}^{2,\alpha}(B)$ for any $0<\alpha<1$. In particular, it holds that $u \in C^2(B)$. Moreover, applying \cite[Theorem 1.3]{SVWZ}, we get that $u \in C^{1,\widetilde{\alpha}}(\overline{B})$ for some $0<\widetilde{\alpha}<\min\{1, 2(1-s)\}$. This obviously implies that $u \in C(\R^N)$. At this point, using the strong maximum principle \cite[Theorem 2.3]{BM}, we obtain that $u>0$ in $B$. Proceeding the classical moving plane method, we then get that $u$ is radially symmetric and strictly decreasing in the radial direction, see \cite[Theorem 4.1]{BM}.

We are now in a position to clarify the non-degeneracy of ground states to \eqref{equ}. Let $u \in H^1_0(B)$ be a ground state to \eqref{equ}. In the first place, we shall demonstrate the non-degeneracy of $u$ in $H^1_{0,r}(B)$. To this end, we are going to introduce the following eigenvalue problem in $H_{0, r}^1(B)$,
\begin{align} \label{equi}
-\Delta w +(-\Delta)^s w -(p-1)u^{p-2} w=\sigma w \quad \mbox{in} \,\,\,B, \quad w=0 \quad \mbox{in} \,\,\, \R^N \backslash {B}.
\end{align}
Taking into account variational arguments, one knows that there exist a sequence of discrete eigenvalues to \eqref{equi} given by
\begin{align} \label{lk}
\sigma_k=\inf_{\substack {S \subset  H^1_{0, r}(B) \\ \mbox{dim}(S)=k}} \sup_{w \in S \backslash \{0\}}\frac{\displaystyle \int_{B} |\nabla w|^2 \,dx + \int_{\R^N} |(-\Delta)^{s/2} w|^2 \,dx}{\displaystyle \int_{B} |w|^2\,dx}, \quad \forall \,\, k \in \N^+.
\end{align}
It is clear that $\sigma_1$ is simple and the corresponding eigenfunction admit constant sign, see Lemma \ref{simple}. Moreover, ones finds that $\sigma_2$ can be alternatively represented by
\begin{align} \label{l2}
\sigma_2=\inf_{\substack{\varphi \in H^1_{0, r}(B) \backslash\{0\} \\ \langle w, w_1 \rangle_2 =0}}\frac{\displaystyle \int_{B} |\nabla w|^2 \,dx + \int_{\R^N} |(-\Delta)^{s/2} w|^2 \,dx}{\displaystyle \int_{B} |w|^2\,dx},
\end{align}
where $w_1>0$ be the first eigenfunction to \eqref{equ1} corresponding to $\sigma_1$. The main result concerning the eigenvalue problem \eqref{equi} reads as follows.

\begin{thm} \label{thm1}
Let $N \geq 2$, $0<s<1$ and $2<p<2^*$. Let $w_2 \in H_{0,r}^1(B)$ be an eigenfunction to \eqref{equi} corresponding to $\sigma_2$. 
Then $\sigma_2$ is simple, $w_2(0) \neq 0$ and $w_2$ changes sign precisely once in the radial direction.
Moreover, $w_2(0)w_2 \mid_{B_{r_0}}$ is decreasing in the radial direction and it holds that
\begin{align} \label{hopf11}
w_2(0)\liminf_{|x| \nearrow 1} \frac{w_2(x)}{1-|x|}<0.
\end{align}
\end{thm}

The core of Theorem \ref{thm1} consists in \eqref{hopf11}, which is the so-called Hopf type lemma and plays an important role in the verification of the non-degeneracy of ground states in $H_{0,r}^1(B)$. To establish \eqref{hopf11}, the first crucial step is to show that $w_2$ changes sign precisely once in the radial direction. Motivated by the elements in \cite{FL, FLS} along with the s-harmonic extension arguments in \cite{CS}, one can  assert that $w_2$ changes sign at most twice in the radial direction, see Lemma \ref{schange4}. In order to further analyze sign properties of $w_2$, we are mainly inspired by the discussion carried out in \cite{FW}. More precisely, adapting the well-known symmetric-decreasing rearrangement arguments in \cite{LL}, we first assert that $w_2(0) \neq 0$, see Lemma \ref{schange1}, which indicates that $\sigma_2$ is simple. It is observed that $w_2$ changes sign precisely once in the radial direction if and only if it holds that
$$
w_2(0) \int_{B} w_2 \,dx<0,
$$
see Lemma \ref{schange2}. Furthermore, it is justified that the second eigenfunction $\varphi_{2,s}$ to the  eigenvalue problem in $H_{0,r}^1(B)$ 
\begin{align} \label{equii}
-\Delta \varphi +(-\Delta)^s \varphi=\lambda \varphi \quad \mbox{in} \,\,\,B, \quad \varphi=0 \quad \mbox{in} \,\,\, \R^N \backslash {B}
\end{align}
changes sign only once in the radial direction and it holds that
$$
\varphi_{2,s}(0) \int_{B} \varphi_{2,s} \,dx<0,
$$
see Lemma \ref{schange3}.
The essence of the proof of Lemma \ref{schange3} is to derive uniform H\"older's estimates of solutions to \eqref{equii} with respect to $s$, which is indeed done by introducing the corresponding extremal operators presented in \cite{CM}. It is well-known that $\varphi_{2,1}$ changes sign only once in the radial direction and it holds that
$$
\varphi_{2,1}(0) \int_{B} \varphi_{2,1} \,dx<0.
$$
Then, applying a continuation argument with respect to $s$, one can prove that the conclusion of Lemma \ref{schange3} holds true. At this point, relying on the two lemmas above and using a continuation argument to the following eigenvalue problem with respect to $\tau$, 
$$
-\Delta w + (-\Delta)^s w -\tau (p-1)u^{p-2} w=\sigma w \quad \mbox{in} \,\,\, B, \quad w=0 \quad \mbox{in} \,\,\, \R^N \backslash {B},
$$
where $0 \leq \tau \leq 1$, we then get that $w_2$ changes sign precisely once in the radial direction, see Lemma \ref{schange0}. This together with a weak maximum principle (see Lemma \ref{mp}) then gives rise to the desired conclusion. In this step, Lemma \ref{mp} is adopted instead of the s-harmonic extension arguments employed in the proof of \cite[Proposition 3.7]{FW}, which is indeed not applicable in the current scenario.
\begin{rem}
Let $V \in L^q(B) \cap C_{loc}^{\alpha}(B)$ be radially symmetric and decreasing in the radial direction for $q>\max \left\{1, N/2 \right\}$ and $0<\alpha<1$. Then, using variational arguments and the fact that $H^1_0(B)$ is compactly embedded into $L^2(B, V dx)$, one can obtain that there exist a sequence of discrete eigenvalues $\{\sigma_k(V)\}$ to the eigenvalue problem in $H_{0,r}^1(B)$
$$
-\Delta w +(-\Delta)^s w +V w=\sigma w \quad \mbox{in} \,\,\,B, \quad w=0 \quad \mbox{in} \,\,\, \R^N \backslash {B},
$$
where
\begin{align*} 
\sigma_k(V):=\inf_{\substack {S \subset  H^1_{0, r}(B) \\ \mbox{dim}(S)=k}} \sup_{w \in S \backslash \{0\}}\frac{\displaystyle \int_{B} |\nabla w|^2 \,dx + \int_{\R^N} |(-\Delta)^{s/2} w|^2 \,dx+ \int_{\R^N} V |w|^2 \,dx}{\displaystyle \int_{B} |w|^2\,dx}, \quad \forall \,\, k \in \N^+.
\end{align*}
Furthermore, employing the analogous arguments developed in the proofs of Lemma \ref{simple} and Theorem \ref{thm1}, one can show that the conclusions of those remain valid.
\end{rem}

With the help of Theorem \ref{thm1}, we are now able to demonstrate the non-degeneracy of ground states in $H_{0,r}^1(B)$. Since $u \in H_{0,r}^1(B)$ is a ground state, by the variational characterizations of $\sigma_1$ and $\sigma_2$ given by \eqref{lk} and \eqref{l2} respectively, then it holds that $\sigma_1 < -\lambda \leq \sigma_2$. As an application of Lemma \ref{sign}, we then conclude that $\sigma_1 < -\lambda < \sigma_2$. which readily infers that $-\lambda$ is not an eigenvalue, i.e. $u$ is non-degenerate in $H_{0,r}^1(B)$. Here the key lies in Lemma \ref{sign}, whose proof is based on an approximation argument inspired partially by the one of \cite[Theorem 5.2]{FW1}. Roughly speaking, to demonstrate Lemma \ref{sign}, we shall argue by contradiction that $\sigma_2=-\lambda$. Let $w$ be an eigenfunction corresponding to $-\lambda$. It is sufficient to verify that
\begin{align} \label{ka}
\int_{B} -\lambda w + (p-1) u^{p-2} w \,dx+\int_{\R^N \backslash \overline{B}} (-\Delta)^s w \,dx+\omega_{N-1} \psi_{w}(1)=0,
\end{align} 
where $\omega_{N-1}$ denotes the surface of the unit sphere in $\R^N$ and
$$
\psi_{w}(1):=\liminf_{|x| \nearrow 1} \frac{w(x)}{1-|x|}.
$$
Indeed, since $w$ is an eigenfunction corresponding to $-\lambda$, then
\begin{align} \label{ka1}
\int_{B}-\Delta w+ (-\Delta)^s w-(p-1)u^{p-2}w\,dx = \int_{B}-\lambda w \,dx.
\end{align}
Thereby, combining \eqref{ka} and \eqref{ka1} and exploiting the fact that $w(0)\psi_{w}(1)<0$ by Theorem \ref{thm1}, we then reach a contradiction. 

To achieve \eqref{ka}, we need to approximate $w$ by $\zeta_k w$, because fractional integration by parts formula is not helpful, where $\zeta_k$ is a proper cut-off function vanishing in a $1/k$-neighborhood of $\partial B$ in $B$. Firstly, we shall integrate the equation satisfied by $\zeta_k w$ over $B$. Later on, introducing the hypergeometric function presented in \cite{MOS} and taking advantage of \cite[Lemma 2.1]{BM} along with \cite[Proposition 6.3 $\&$ Lemma 6.8]{DFW}, we need to estimate the integrals involving the terms $(-\Delta)^s (\zeta_k w)$, $\Delta (\zeta_k w)$ and $\nabla (\zeta_k w)$. This is the subtle part in the discussion. Finally, passing the limit as $k \to \infty$, we then get the desired conclusion.

It is worth pointing out that since the problem under our consideration possesses the mixed operators in $B \subset \R^N$ and $N \geq 2$, then the analysis we conducted here is different from the one in \cite{FW1}, where the non-degeneracy of solutions to \eqref{equ0} was considered for $p=3$, $1/6<s<1$ and $B=(-1, 1) \subset \R$. Let us also mention that fractional integration by parts formula, which is prominent in the corroboration of the non-degeneracy of solutions to \eqref{equ0} (see \cite{DP, FW, FW1}), is not available in our case yet, due to the presence of the mixed operators.  

Next, we are going to investigate the non-degeneracy of ground states to \eqref{equ} in $H_{0,nr}^1(B):=H_{0}^1(B) \backslash H_{0,r}^1(B)$. In this situation, we shall introduce the following eigenvalue problem in $H_{0,nr}^1(B)$,
\begin{align} \label{equi1}
-\Delta w + (-\Delta)^s w + \lambda w=\Lambda u^{p-1} w \quad \mbox{in} \,\,\, B, \quad w=0 \quad \mbox{in} \,\,\, \R^N \backslash {B}.
\end{align}
It is not hard to check that there exist a sequence of discrete eigenvalues to \eqref{equi1} with the aid of variational arguments. Let $\Lambda_1$ and $\Lambda_2$ be the first eigenvalue and the second eigenvalue to \eqref{equi1} respectively. Clearly, one finds that $\Lambda_1=1$, because $u \in H_0^1(B)$ is a positive solution to \eqref{equ}. To derive the desired conclusion, it is sufficient to prove that $\Lambda_2>p-1$, which then indicates that $\Lambda_1<p-1<\Lambda_2$. For this, we shall argue by contradiction that $\Lambda_2 \leq p-1$. Let $w_2 \in H_{0,nr}^1(B)$ be an eigenfunction to \eqref{equi1} corresponding to $\Lambda_2$. Since $w_2 \in H_{0,nr}^1(B)$, then we can construct an eigenfunction to \eqref{equi1} corresponding to $\Lambda_2$, which is antisymmetric with respect to the reflection at certain hyperplane containing the origin.
Later on, we principally make use of a Picone type identity concerning antisymmetric functions to rule out the existence of antisymmetric eigenfunctions to \eqref{equi1} corresponding to $\Lambda_2$ provided $\Lambda_2 \leq p-1$. Therefore, we reach a contradiction and the desired conclusion follows. Consequently, we obtain the following result.

\begin{thm} \label{thm2}
Let $N \geq 2$, $0<s<1$, $\lambda>-\lambda_1$ and $2<p<2^*$. Let $u \in H^1_0(B)$ be a ground state to \eqref{equ}. Then it is non-degenerate in $H^1_0(B)$, i.e. the following linearized equation only has the trivial solution $w=0$ in $H^1_0(B)$,
$$
-\Delta w +(-\Delta)^s w -(p-1)u^{p-2} w=0.
$$
\end{thm}

Presently, depending on the non-degeneracy of ground states to \eqref{equ}, we are able to verify the uniqueness of ground states. To begin with, taking into account the classical blow-up argument originating from \cite{GS} and proceeding a delicate asymptotic analysis with the assistance of the regularity theory and Hopf's lemma of solutions to elliptic equations with the mixed operators obtained in \cite{BM, BMS, SVWZ}, we are able to establish the following result.

\begin{thm} \label{thm3}
Let $N \geq 2$, $0<s<1$, $\lambda>-\lambda_1$ and $2<p<2^*$. Then there exists $p_0>2$ such that \eqref{equ} has a unique ground state in $H_0^1(B)$ for any $2<p<p_0$.
\end{thm}

At this point, invoking Theorems \ref{thm2}-\ref{thm3}, the regularity theory of solutions to \eqref{equ} established in \cite{BDVV, SVWZ} and a continuation argument with respect to $p$ relying on the well-known implicit function theorem, we then have the following result.

\begin{thm} \label{thm4}
Let $N \geq 2$, $0<s<1$, $\lambda>-\lambda_1$ and $2<p<2^*$. Then \eqref{equ} possesses a unique ground state in $H^1_0(B)$.
\end{thm}

The paper is organized as follows. In Section \ref{non-degeneracy}, we discuss the non-degeneracy of ground states to \eqref{equ} and establish Theorems \ref{thm1} and \ref{thm2}. In Section \ref{uniqueness}, we consider the uniqueness of ground states to \eqref{equ} and prove Theorems \ref{thm3} and \ref{thm4}.

\section{Non-degeneracy of ground states} \label{non-degeneracy}

The principal aim of this section is to discuss the non-degeneracy of ground states in $H_0^1(B)$. Let $u \in H_{0}^1(B)$ be a ground state to \eqref{equ}. We know that $u>0$ in $B$ and $u \in L^{\infty}(B)$. Moreover, it is radially symmetric and strictly decreasing in the radial direction. Define $V:=-(p-1)u^{p-2}<0$ in $B$. To begin with, we consider the non-degeneracy of the ground state in $H_{0,r}^1(B)$. 

\subsection{Non-degeneracy of ground states in $H_{0,r}^1(B)$}
Let us introduce the following eigenvalue problem in $H_{0,r}^1(B)$,
\begin{align} \label{equ1}
-\Delta w +(-\Delta)^s w + V w=\sigma w \quad \mbox{in} \,\,\,B, \quad w=0 \quad \mbox{in} \,\,\, \R^N \backslash {B}.
\end{align}
Let $\sigma_1$ denote the first eigenvalue to \eqref{equ1} in $H_{0,r}^1(B)$. It can be represented by
\begin{align} \label{min1}
\begin{split}
\sigma_1:&=\inf_{w \in H^1_{0, r}(B) \backslash \{0\}} \frac{\displaystyle\int_{B} |\nabla w|^2 \,dx + \int_{\R^N} |(-\Delta)^{\frac s 2} w|^2 \,dx + \int_{B} V |w|^2 \,dx}{\displaystyle \int_{B} |w|^2\,dx} \\
&=\inf_{\substack{w \in H^1_{0, r}(B) \backslash \{0\} \\ \|w\|_2=1}} \displaystyle \int_{B} |\nabla w|^2 \,dx + \int_{\R^N} |(-\Delta)^{\frac s 2} w|^2 \,dx + \int_{B} V |w|^2 \,dx.
\end{split}
\end{align}
Since $V \in L^{\infty}(B)$, then $\sigma_1>-\infty$. Let $\{w_n\} \subset H_{0,r}^1(B)$ and $\|w_n\|_2=1$ be a minimizing sequence to \eqref{min1}.
Since $H_{0,r}^1(B)$ is compactly embedded into $L^2(B)$, then $\{w_n\}$ is compact in $H^1_{0,r}(B)$. It then follows that $\sigma_1$ is achieved by some $w_1 \in H^1_{0,r}(B) \backslash \{0\}$. Observe that $|w_1| \in H^1_{0,r}(B)$ is also a minimizer to \eqref{min1}. Therefore, without restriction, we may assume that $w_1 \geq 0$ in $B$ and satisfy the equation
$$
-\Delta w_1 +(-\Delta)^s w_1 + V w_1=\sigma_1 w_1 \quad \mbox{in} \,\,\, B, \quad w_1=0 \quad \mbox{in} \,\,\, \R^N \backslash {B}.
$$
By the strong maximum principle \cite[Theorem 2.3]{BM}, we then obtain that $w_1>0$ in $B$.

\begin{lem} \label{simple}
Let $\sigma_1$ be the first eigenvalue to \eqref{equ1}. Then it is simple.
\end{lem}
\begin{proof}
Let $w_1, \widetilde{w}_1 \in H^1_{0,r}(B)$ and $\|w_1\|_2=\|\widetilde{w}_1\|_2=1$ be two positive eigenfunctions corresponding to $\sigma_1$. Define
$$
w:=\left(\frac{w_1^2+\widetilde{w}_1^2}{2}\right)^{\frac 12}.
$$
It is not hard to find that $w \in H^1_{0,r}(B)$ and $\|w\|_2=1$. Observe that the function $(s, t) \mapsto |s^{\frac 12}-t^{\frac 12}|^2$ is convex on $[0, +\infty) \times [0, +\infty)$. Therefore, we conclude that
\begin{align} \label{convex}
|w(x)-w(y)|^2 \leq \frac 12 |w_1(x)-w_1(y)|^2+\frac 12 |\widetilde{w}_1(x)-\widetilde{w}_1(y)|^2, \quad \forall \,\, x, y \in \R^N,
\end{align}
where the equality holds if and only if $w_1(x)\widetilde{w}_1(y)=w_1(y)\widetilde{w}_1(x)$. Furthermore, we find that
$$
|\nabla w|^2=\frac 14 \left(\frac{w_1^2+\widetilde{w}_1^2}{2}\right)^{-1} |w_1 \nabla w_1+\widetilde{w}_1 \nabla \widetilde{w}_1|^2 \leq \frac 12 \left(|\nabla w_1|^2+|\nabla \widetilde{w}_1|^2\right).
$$
As a result, applying \eqref{defn}, we obtain that
\begin{align*}
\sigma_1 &\leq \int_{B} |\nabla w|^2 \,dx + \int_{\R^N} |(-\Delta)^{\frac s 2} w|^2 \,dx + \int_{B} V |w|^2 \,dx \\
& \leq \frac 12 \left(\int_{B} |\nabla w_1|^2 \,dx + \int_{\R^N} |(-\Delta)^{\frac s 2} w_1|^2 \,dx + \int_{B} V |w_1|^2 \,dx \right. \\
& \qquad \left. +\int_{B} |\nabla \widetilde{w}_1|^2 \,dx + \int_{\R^N} |(-\Delta)^{\frac s 2} \widetilde{w}_1|^2 \,dx + \int_{B} V |\widetilde{w}_1|^2 \,dx \right) =\sigma_1.
\end{align*}
This implies that the equality in \eqref{convex} holds true. Therefore, we get that there exists a constant $C>0$ such that $w_1=C\widetilde{w}_1$ in $\R^N$ and the proof is complete.
\end{proof}

Let $w_1 \in H^1_{0,r}(B)$ be the first eigenfunction corresponding to $\sigma_1$. Let $\sigma_2$ denote the second eigenvalue to \eqref{equ1} in $H_{0,r}^1(B)$. It can be represented by
\begin{align} \label{min2}
\sigma_2:&=\inf_{\substack{w \in H^1_{0, r}(B) \backslash \{0\} \\ \langle w, w_1 \rangle_2=0}} \frac{\displaystyle\int_{B} |\nabla w|^2 \,dx + \int_{\R^N} |(-\Delta)^{\frac s 2} w|^2 \,dx + \int_{B} V |w|^2 \,dx}{\displaystyle \int_{B} |w|^2\,dx}\\
&=\inf_{\substack{w \in H^1_{0, r}(B) \backslash \{0\} \\ \langle w, w_1 \rangle_2=0, \|w\|_2=1}} \displaystyle \int_{B} |\nabla w|^2 \,dx + \int_{\R^N} |(-\Delta)^{\frac s 2} w|^2 \,dx + \int_{B} V |w|^2 \,dx.
\end{align}
Analogously, reasoning as previously, one can obtain that $\sigma_2$ can be achieved by some $w_2 \in H^1_{0,r}(B) \backslash \{0\}$. Since $\langle w_2, w_1 \rangle_2=0$ and $w_1>0$ in $B$, then $w_2$ is sign-changing in $B$. 

Define $\R^{N+1}_+:=\left\{(x, t) \in \R^N \times \R : t>0\right\}$. Let $W$ denote the s-harmonic extension of $w$ in the spirit of the arguments developed in \cite{CS}, which is given by 
\begin{align} \label{defw}
W(x ,t)=P(N, s) t^{2s} \int_{\R^N} \frac{w(y)}{\left(t^2+ |x-y|^2\right)^{\frac{N+2s}{2}}}\,dy, \quad \frac{1}{P(N, s)}:= \int_{\R^N} \frac{1}{\left(1+ |x|^2\right)^{\frac{N+2s}{2}}}\,dx.
\end{align}
Moreover, we know that $W$ solves the equation
\begin{align} \label{equs}
\left\{
\begin{aligned}
&\mbox{div} \left(t^{1-2s} \nabla W\right)=0 \quad &\mbox{in} \,\,\, \R^{N+1}_+, \\
&W(x, 0)=w \quad &\mbox{in} \,\,\, \R^{N}, \\
&d_s \lim_{t \to 0} t^{1-2s} \partial_t W=-(-\Delta)^{s} w \quad &\mbox{in} \,\,\, \R^{N},
\end{aligned}
\right.
\end{align}
where
$$
d_s:=2^{2s-1}\frac{\Gamma(s)}{\Gamma(1-s)}.
$$

Let $D^{1,2}(\R^{N+1}_+; t^{1-2s})$ be the completion of $C_0^{\infty}(\R^{N+1}_+)$ with respect to the norm
$$
\|U\|_{D^{1,2}}:=\left(\int_{\R^{N+1}_+}|\nabla U|^2 t^{1-2s} \,dxdt \right)^{\frac 12}.
$$
Define
$$
\mathcal{H}^1_0(\R^{N+1}_+; t^{1-2s}):=\left\{U \in D^{1,2}(\R^{N+1}_+, t^{1-2s}): U(\cdot, 0) \in H^1_0(B)\right\}
$$
equipped with the norm
$$
\|U\|_{\mathcal{H}_0^1}:=\|U\|_{D^{1,2}} + \|U(\cdot, 0)\|_{H^1_0}.
$$
\begin{lem} \label{ineq}
It holds that
$$
\int_{\R^N} |(-\Delta)^{\frac s 2} U(\cdot, 0)|^2 \,dx \leq d_s \int_{\R^{N+1}_+} |\nabla U|^2 t^{1-2s}\,dxdt, \quad \forall \,\, U \in D^{1,2}(\R^{N+1}_+, t^{1-2s}),
$$
where the equality holds if and only if $U$ is the s-harmonic extension of some function $f \in D^{1,2}(\R^N)$.
\end{lem}
\begin{proof}
The assertion follows by a direct adaptation of the arguments presented in the proof of \cite[Proposition 3.6]{FL} (see also \cite[Proposition 5.1]{FLS}). Hence we omit the proof.
\end{proof}

\begin{lem} \label{eign}
Let $\sigma_2$ be the second eigenvalue to \eqref{equ1}. Then it holds that
$$
\sigma_2=\inf_{\substack{W \in M \\ \|W(\cdot, 0)\|_2=1}}\int_{B} |\nabla W(\cdot, 0)|^2 \,dx + d_s\int_{\R^{N+1}_+} |\nabla W|^2 t^{1-2s} \,dxdt + \int_{B} V |W(\cdot, 0)|^2 \,dx,
$$
where
$$
M:=\left\{ U \in \mathcal{H}^1_0(\R^{N+1}_+; t^{1-2s}) \backslash \{0\} : \int_{B} U(\cdot, 0) W_1(\cdot, 0) \,dx=0, U(\cdot, t) \,\,\mbox{is radially symmetric for any} \,\, t \geq 0 \right\}
$$ 
and $W_1$ is the s-harmonic extension of the first eigenfunction $w_1$ to \eqref{equ1}.
\end{lem}
\begin{proof}
Let $w_2 \in H^1_{0,r}(B)$ and $\|w_2\|_2=1$ be an eigenfunction to \eqref{equ1} corresponding to $\sigma_2$. Let $W_2 \in \mathcal{H}^1_0(\R^{N+1}_+; t^{1-2s})$ be the s-harmonic extension of $w_2$. It then follows from \eqref{defw} that $W_2 \in M$, because $w_2$ is radially symmetric. By Lemma \ref{ineq}, we then get that
\begin{align} \label{s2}
\begin{split}
\sigma_2 &=\int_{B} |\nabla w_2|^2 \,dx + \int_{\R^N} |(-\Delta)^{\frac s 2} w_2|^2 \,dx + \int_{B} V |w_2|^2 \,dx \\ 
&=\int_{B} |\nabla W_2(\cdot, 0)|^2 \,dx + d_s\int_{\R^{N+1}_+} |\nabla W_2|^2 t^{1-2s} \,dxdt + \int_{B} V |W_2(\cdot, 0)|^2 \,dx \\
&\geq \inf_{\substack{W \in M \\ \|W(\cdot, 0)\|_2=1}}\int_{B} |\nabla W(\cdot, 0)|^2 \,dx + d_s\int_{\R^{N+1}_+} |\nabla W|^2 t^{1-2s} \,dxdt + \int_{B} V |W(\cdot, 0)|^2 \,dx.
\end{split}
\end{align}
On the other hand, by Lemma \ref{ineq}, we know that, for any $W \in M$ and $\|W(\cdot, 0)\|_2=1$, it holds that
\begin{align*}
\sigma_2 &\leq \int_{B} |\nabla W(\cdot, 0)|^2 \,dx+ \int_{\R^N} |(-\Delta)^{\frac s 2} W(\cdot, 0)|^2 \,dx +\int_{B} V |W(\cdot, 0)|^2 \,dx  \\
&\leq \int_{B} |\nabla W(\cdot, 0)|^2 \,dx + d_s\int_{\R^{N+1}_+} |\nabla W|^2 t^{1-2s} \,dxdt + \int_{B} V |W(\cdot, 0)|^2 \,dx,
\end{align*}
which clearly infers that
\begin{align} \label{s3}
\sigma_2 \leq \inf_{\substack{W \in M \\ \|W(\cdot, 0)\|_2=1}}\int_{B} |\nabla W(\cdot, 0)|^2 \,dx + d_s\int_{\R^{N+1}_+} |\nabla W|^2 t^{1-2s} \,dxdt + \int_{B} V |W(\cdot, 0)|^2 \,dx.
\end{align}
Combining \eqref{s2} and \eqref{s3} then leads to the desired conclusion. This completes the proof.
\end{proof}

\begin{lem} \label{nodal}
Let $W_2$ be the s-harmonic extension of $w_2$, where $w_2 \in H^1_{0, r}(B)$ is an eigenfunction to \eqref{equ1} corresponding to $\sigma_2$. Then $W_2$ has precisely two nodal domains  in $\R^{N+1}_+$. 
\end{lem}
\begin{proof}
Since $w_2 \in H^1_{0, r}(B)$ is an eigenfunction corresponding to $\sigma_2$, then it satisfies the equation
$$
-\Delta w_2 + (-\Delta)^s w_2 +V(x)w_2=\sigma_2 w_2 \quad \mbox{in} \,\,\, B, \quad w_2=0 \quad \mbox{in} \,\,\, \R^N \backslash {B}
$$
This then shows that $W_2 \in \mathcal{H}^1_0(\R^{N+1}_+; t^{1-2s})$ solves the equation
\begin{align} \label{equw2}
\left\{
\begin{aligned}
&\mbox{div} \left(t^{1-2s} \nabla W_2\right)=0 \quad &\mbox{in} \,\,\, \R^{N+1}_+, \\
&W_2(x, 0)=w_2 \quad &\mbox{in} \,\,\, \R^{N}, \\
&d_s \lim_{t \to 0} t^{1-2s} \partial_t W_2=-\Delta w_2+V(x)w_2-\sigma_2 w_2 \quad &\mbox{in} \,\,\, \R^{N}.
\end{aligned}
\right.
\end{align} 
First we are going to show that $W_2$ had at most two nodal domains in $\R^{N+1}_+$. For this, we shall employ the elements developed in the proof of \cite[Proposition 5.2]{FLS}. We shall suppose by contradiction that $W_2$ has three nodal domains $\Omega_1, \Omega_2$ and $\Omega_3$ in $\R^{N+1}_+$. Since $w_2 \in L^{\infty}(B) \cap C(\R^N)$, by \cite[Remark 3.8]{CaSi}, then $W_2 \in L^{\infty}(\R^{N+1}_+) \cap C(\overline{\R^{N+1}_+})$. 
It then follows that there exists $\overline{\Omega}_i \subset \overline{\R^{N+1}_+}$ such that $\overline{\Omega}_i \cap \partial\R^{N+1}_+ \neq \emptyset$ for some $i=1,2,3$. Without restriction, we may assume that $i=1$. Define 
$$
W:= W_2 (\gamma_1\chi_{\overline{\Omega}_1} + \gamma_2 \chi_{\overline{\Omega}_2}) \in \mathcal{H}^1_0(\R^{N+1}_+; t^{1-2s}),
$$ 
where $\gamma_1, \gamma_2 \in \R$ are constants to be determined later. Observe that $
\nabla W=\nabla W_2 (\gamma_1\chi_{\overline{\Omega}_1} + \gamma_2 \chi_{\overline{\Omega}_2})$.
Let us now take $\gamma_1, \gamma_2 \in \R$ be such that $\|W(\cdot, 0)\|_2=1$ and $W \in M$.
Define 
$$
\widetilde{W}:=W_2 (\gamma_1^2\chi_{\overline{\Omega}_1} + \gamma_2^2 \chi_{\overline{\Omega}_2}) \in \mathcal{H}^1_0(\R^{N+1}_+; t^{1-2s}).
$$ 
Testing \eqref{equw2} by $\widetilde{W}$, we have that
\begin{align*}
\int_{\R^N}  \nabla \widetilde{W}( \cdot, 0) \cdot \nabla w_2\,dx +d_s\int_{\R^{N+1}_+} \nabla W_2 \cdot \nabla \widetilde{W} t^{1-2s} \,dxdt + \int_{\R^N} V \widetilde{W}( \cdot, 0) w_2\,dx =\sigma_2 \int_{\R^N} w_2 \widetilde{W}(\cdot, 0) \,dx.
\end{align*}
Note that $\nabla \widetilde{W}=\nabla W_2 (\gamma_1^2\chi_{\overline{\Omega}_1} + \gamma_2^2 \chi_{\overline{\Omega}_2})$.
As a consequence, we conclude that
$$
\int_{\R^N}  |\nabla W(\cdot, 0)|^2\,dx + d_s\int_{\R^{N+1}_+} |\nabla W|^2 t^{1-2s} \,dxdt+\int_{\R^N} V | W(\cdot, 0)|^2\,dx =\sigma_2 \int_{\R^N} |W(\cdot, 0)|^2 \,dx=\sigma_2.
$$
On the other hand, since $\|W(\cdot, 0)\|_2=1$ and $W \in M$, then 
$$
\sigma_2 \leq \int_{\R^N}  |\nabla W(\cdot, 0)|^2\,dx + \int_{\R^N} |(-\Delta)^{\frac s 2} W(\cdot, 0)|^2 \,dx+\int_{\R^N} V | W(\cdot, 0)|^2\,dx.
$$
It then follows from Lemma \ref{ineq} that
$$
d_s\int_{\R^{N+1}_+} |\nabla W|^2 t^{1-2s} \,dxdt=\int_{\R^N} |(-\Delta)^{\frac s 2} W(\cdot, 0)|^2 \,dx.
$$
As an application of Lemma \ref{ineq}, we then derive that $W$ is the s-harmonic extension of some function $f \in D^{1,2}(\R^N)$ and it solves the equation
\begin{align} \label{equw}
\mbox{div} (t^{1-2s} \nabla W)=0 \quad \mbox{in} \,\,\, \R^{N+1}_+.
\end{align}
Furthermore, by the definition of $W$, we find that $W = 0$ in $\Omega_3$. Using the unique continuation to \eqref{equw}, we then get that $W=0$ in $\R^{N+1}_+$. This is impossible. It then shows that $W_2$ possesses at most two nodal domains in $\R^{N+1}_+$. Observe that $W_2(\cdot, 0)=w_2$ changes sign in $B$ and $W_2 \in C(\overline{\R^{N+1}_+})$.  It then follows that $W_2$ has exactly two nodal domains in $\R^{N+1}_+$. Thus the proof is complete. 
\end{proof}

\begin{lem} \label{schange4}
Let $w_2 \in H^1_{0, r}(B)$ be an eigenfunction corresponding to $\sigma_2$. Then $w_2$ changes sign at most twice in the radial direction. 
\end{lem}
\begin{proof}
We shall argue by contradiction. If not, we may assume that $w_2$ changes sign at least three times in the radial direction. Therefore, we may suppose that there exist $0<r_1<r_2<r_3<r_4<1$ such that
\begin{align*} 
w_2(r_1)>0, \quad w_2(r_3)>0, \quad w_2(r_2)<0, \quad w_2(r_4)<0.
\end{align*}
Let $W_2$ be the s-harmonic extension of $w_2$. From Lemma \ref{nodal}, we conclude that $W_2$ has precisely two nodal domains in $\R^{N+1}_+$ denoted by $\Omega_1$ and $\Omega_2$, where
$$
\Omega_1:=\left\{(x ,t) \in \overline{\R^{N+1}_+} : W_2(x ,t)>0\right\}, \quad \Omega_2:=\left\{(x ,t) \in \overline{\R^{N+1}_+} : W_2(x ,t)<0\right\}.
$$
Since $w_2$ is radially symmetric, then $W_2$ is radially symmetric with respect to $x$ by \eqref{defw}. It then infers that $\Omega_1$ and $\Omega_2$ are radially symmetric with respect to $x$. 
Define
\begin{align} \label{defo}
\widetilde{\Omega}_{i}:=\left\{(|x|, t) \in \overline{\R^2_+} : (x, t) \in \Omega_{i}\right\}, \quad i=1,2.
\end{align}
Since $\widetilde{\Omega}_1$ and $\widetilde{\Omega}_2$ are path connected, then there exist $\eta_1 \in C([0, 1], \widetilde{\Omega}_{1})$ and $\eta_2 \in C([0, 1], \widetilde{\Omega}_{2})$ such that 
$$
\eta_1(0)=(r_1, 0), \quad \eta_1(1)=(r_3, 0), \quad \eta_2(0)=(r_2, 0), \quad \eta_2(1)=(r_4, 0).
$$
At this point, applying \cite[Lemma 6.1]{FW}, we obtain that $\eta_1$ and $\eta_2$ intersect in $\R^2_+$. This is impossible. Therefore, the desired conclusion follows and
the proof is complete.
\end{proof}

\begin{lem} \label{schange}
Let $w \in H^1_0(B)$ be an eigenfunction to \eqref{equ1} corresponding to some eigenvalue $\sigma$. Let $W$ be the s-harmonic extension of $w$. If there exists $x_0 \in B$ such that $w(x_0)=0$. Then $W$ changes sign in any relative neighborhood of $(x_0, 0)$ in $\overline{\R^{N+1}_+}$. 
\end{lem}
\begin{proof}
We shall argue by contradiction. Without restriction, we may assume that there exists a relative neighborhood of $(x_0, 0)$ in $\overline{\R^{N+1}_+}$ denoted by $U$ such that $W \geq 0$ in $U$. Since $W$ is the s-harmonic extension of $w$, then
\begin{align} \label{equw1}
\mbox{div} (t^{1-2s} \nabla W)=0 \quad \mbox{in} \,\,\, \R^{N+1}_+.
\end{align}
Therefore, we have that $W \neq 0$ in $U$. If not, by the unique continuation to \eqref{equw1}, it holds that $W=0$ in $\R^{N+1}_+$. This clearly leads to $w=0$ in $B$, because of $W \in C(\overline{\R^{N+1}_+})$. We then reach a contradiction. Since $W \geq 0$ in $U$, by the strong maximum principle, we then get that $W>0$ in $U \cap \R^{N+1}_+$. Since $w(x_0)=W(x_0, 0)=0$, by Hopf's Lemma (\cite[Proposition 4.11]{CaSi}), then
\begin{align} \label{hopf1}
-\lim_{t \to 0} t^{1-2s} \partial_t W(x_0,0)<0.
\end{align}
Noting that $w \in H^1_0(B)$ is a solution to the equation
\begin{align*} 
-\Delta w  + (-\Delta)^s w  + V w=\sigma w \quad \mbox{in} \,\,\, B, \quad w=0 \quad \mbox{in} \,\,\, \R^N \backslash B.
\end{align*}
we then have that
\begin{align} \label{hfw}
-d_s \lim_{t \to 0} t^{1-2s} \partial_t  W(x_0, 0)=(-\Delta)^s w(x_0)=\Delta w(x_0) -V(x_0) w(x_0)+ \sigma w(x_0)=\Delta w(x_0) .
\end{align}
This along with \eqref{hopf1} implies that $\Delta w(x_0)<0$. Since $w \in C^2(B)$, then there exists $0<\delta<1$ such that $\Delta w<0$ in $B_{\delta}(x_0)$. Moreover, since $W>0$ in $U \cap \R^{N+1}_+$ and $W \in C(\overline{\R^{N+1}_+})$, then $w \geq 0$ in $B_{\delta}(x_0)$. Thanks to $w(x_0)=0$, by the well-known strong maximum principle in \cite{GT}, we then derive that $w=0$ in $B_{\delta}(x_0)$. As a consequence, it holds that $\Delta w(x_0)=0$. Going back to \eqref{hfw}, we then get that 
$$
-d_s t^{1-2s} \partial_t  W(x_0, 0)=0.
$$ 
Obviously, this contradicts with \eqref{hopf1}.
Therefore, the desired conclusion follows and the proof is complete.
\end{proof}

\begin{lem} \label{schange11}
Let $w_2 \in H^1_{0, r}(B)$ be an eigenfunction to \eqref{equ1} corresponding to $\sigma_2$. If $w_2(0)=0$ and $w_2 \neq 0$ in $B$, then $w_2$ changes sign only once.
\end{lem}
\begin{proof}
We shall argue by contradiction. Since $w_2 \in H^1_{0,r}(B)$ changes sign at least once, then we may assume by contradiction that there exist $0<r_1<r_2<r_3<1$ such that 
$$
w_2(r_1)>0, \quad w_2(r_3)>0, \quad w_2(r_2)<0.
$$
Let $W_2$ be the s-harmonic extension of $w_2$. Since $W_2$ is radially symmetric with respect to $x$, then we define $\widetilde{W}_2 : \overline{\R^2_+} \to \R$ by $\widetilde{W}_2 (|x|, t):=W_2(x, t)$. Observe that the nodal domain $\widetilde{\Omega}_1$ of $\widetilde{W}_2$ in $\overline{\R^2_+}$ are path connected, where $\widetilde{\Omega}_1$ is defined by \eqref{defo}. Therefore, there exists $\eta_1 \in C([0, 1], \widetilde{\Omega}_1)$ such that $\eta_1(0)=(r_1, 0)$ and $\eta_1(1)=(r_3, 0)$. Since $w_2(0)=\widetilde{W}_2(0, 0)=0$, then $(0, 0) \not \in \eta_1([0, 1])$. Define
$$
d:=\mbox{dist}(\eta_1([0, 1]), (0, 0)).
$$
It holds that $d>0$. If not, then there exists $\{t_n\} \subset [0, 1]$ such that $\eta_1(t_n) \to (0, 0)$ as $n \to \infty$. This implies that there exists $t_0 \in [0, 1]$ such that $\eta_1(t_0)=(0, 0)$, which is impossible. Since $w_2(0)=0$ and $\widetilde{W}_2$ changes sign in any relative neighborhood of $(0, 0)$ in $\overline{\R^2_+}$ by Lemma \ref{schange}, then there exists $z \in \overline{\R^2_+}$ such that $|z|<d$ and $\widetilde{W}_2(z)<0$. It yields that $z \in \widetilde{\Omega}_2$, where $\widetilde{\Omega}_2$ is defined by \eqref{defo}. Since the nodal domain $\widetilde{\Omega}_2$ of $\widetilde{W}_2$ in $\overline{\R^2_+}$ is path connected as well, then there exists $\eta_2 \in C([0, 1], \widetilde{\Omega}_2)$ such that $\eta_1(0)=z$ and $\eta_2(1)=(r_2, 0)$. It then follows from \cite[Lemma 6.2]{FW} that $\eta_1$ and $\eta_2$ intersect in $\R^2$. This is impossible. It in trun suggests that $w_2$ changes sign only once and the proof is complete.
\end{proof}

\begin{lem} \label{schange1}
Let $w_2 \in H^1_{0, r}(B)$ be an eigenfunction to \eqref{equ1} corresponding to $\sigma_2$. If $w_2(0)=0$, then $w_2=0$ in $B$. In particular, it holds that $\sigma_2$ is simple.
\end{lem}
\begin{proof}
On the contrary, we shall assume that $w_2 \neq 0$ in $B$. By Lemma \ref{schange11}, we then get that $w_2$ changes sign only once. Therefore, without restriction, we may assume that there exists $0<r_0<1$ such that $w_2 \geq 0$ in $B_{r_0}(0)$ and $w_2 \leq 0$ in $B \backslash B_{r_0}$. In addition, we know that $w_2^+ \neq 0$ and $w_2^- \neq 0$, where
$$
w_2^+:=\max\left\{w_2(x), 0\right\}, \quad w_2^-:=\max\left\{-w_2(x), 0\right\}, \quad w_2=w_2^+-w_2^-.
$$ 
Let $(w_2^+)^*$ denote the symmetric-decreasing rearrangement of $w_2^+$. It follows that $\mbox{supp} \,\, (w_2^+)^* \subset B_{r_0}$ and $\mbox{supp} \,\, (w_2^+)^* \cap \mbox{supp} \,\, w_2^-=\emptyset$. Thanks to \cite[Theorem 3.7 $\&$ Lemma 7.17]{LL}, we then have that
\begin{align} \label{r1}
\int_{\R^N} |\nabla (w_2^+)^*|^2 \,dx \leq \int_{\R^N} |\nabla w_2^+|^2 \,dx, \quad  [(w_2^+)^*]_s^2 \leq [w_2^+]_s^2.
\end{align}
Furthermore, it holds that $\|w_2^+\|_2=\|(w_2^+)^*\|_2$ and $((w_2^+)^2)^*=((w_2^+)^*)^2$. Since $V \leq 0$ in $\R^N$ and it is radially symmetric, by \cite[Theorem 3.4]{LL}, then
$$
-\int_{\R^N} V (w_2^+)^2 \,dx \leq -\int_{\R^N} V((w_2^+)^*)^2 \,dx.
$$
It then leads to
\begin{align} \label{r2}
\int_{\R^N} (\sigma_2-V) (w_2^+)^2 \,dx \leq \int_{\R^N} (\sigma_2-V) ((w_2^+)^*)^2 \,dx.
\end{align}
Since $w_2 \in H^1_{0,r}(B)$ solves \eqref{equ1} with $\sigma=\sigma_2$, by testing \eqref{equ1} with $w_2^+$ and $w_2^-$ respectively, then we get that
\begin{align} \label{r3}
\begin{split}
\int_{\R^N} |\nabla w_2^+|^2 \,dx + \int_{\R^N} (V - \sigma_2) (w_2^+)^2 \,dx &= [w_2^+, w_2^-]_s-[w_2^+]_s^2, \\
\int_{\R^N} |\nabla w_2^-|^2 \,dx + \int_{\R^N} (V - \sigma_2) (w_2^-)^2 \,dx &= [w_2^+, w_2^-]_s-[w_2^-]_s^2,
\end{split}
\end{align}
where
$$
[w^+_2, w_2^-]_s:=\frac{C(N,s)}{2} \int_{\R^N} \int_{\R^N} \frac{\left(w^+_2(x)-w^+_2(y)\right)\left(
w^-_2(x)-w^-_2(y)\right)}{|x-y|^{N+2s}} \,dxdy.
$$
Let $\kappa>0$ be such that
$$
\int_{B} \left((w_2^+)^*-\kappa w_2^- \right)w_1 \,dx=0.
$$
Note that $[w_2^+, w_2^-]_s \leq [(w_2^+)^*, w_2^-]_s$, which has been established in the proof of \cite[Theorem 3.2]{FW}. Therefore, by \eqref{min2} and \eqref{r1}, we derive that
\begin{align} \label{w220} 
\int_{B} \left(\sigma_2-V\right)\left((w_2^+)^*-\kappa w_2^- \right)^2 \,dx 
&\leq \int_{B} |\nabla \left((w_2^+)^*-\kappa w_2^-\right)|^2 \,dx + [(w_2^+)^*-\kappa w_2^-]_s^2 \\ \nonumber
& =\int_{B} |\nabla (w_2^+)^*|^2 \,dx+ \kappa^2 |\nabla w_2^-|^2 \,dx +[(w_2^+)^*]_s^2+\kappa^2 [w_2^-]_s^2-2\kappa [(w_2^+)^*, w_2^-]_s \\  \nonumber
& \leq \int_{B} |\nabla w_2^+|^2 \,dx+ \kappa^2  |\nabla w_2^-|^2 \,dx +[w_2^+]_s^2+\kappa^2 [w_2^-]_s^2-2\kappa [w_2^+, w_2^-]_s.  \nonumber 
\end{align}
Combining this with \eqref{r2} and \eqref{r3}, we then get that
\begin{align*}
\int_{B} \left(\sigma_2-V\right)\left(w_2^+-\kappa w_2^- \right)^2 \,dx
&=\int_{B} \left(\sigma_2-V\right)\left((w_2^+)^2+\kappa^2(w_2^-)^2 \right) \,dx \\
& \leq \int_{B} \left(\sigma_2-V\right)\left(((w_2^+)^*)^2+\kappa^2(w_2^-)^2 \right) \,dx \\
&=\int_{B} \left(\sigma_2-V\right)\left((w_2^+)^*-\kappa w_2^- \right)^2 \,dx \\
&=\int_{B} \left(\sigma_2-V\right)\left(w_2^+-\kappa w_2^- \right)^2 \,dx +(1-\kappa)^2[w_2^+, w_2^-]_s.
\end{align*}
Since $[w_2^+, w_2^-]_s<0$, then $\kappa=1$ and 
\begin{align} \label{w210}
\int_{B} \left(\sigma_2-V\right)\left(w_2^+-w_2^- \right)^2 \,dx=\int_{B} \left(\sigma_2-V\right)\left((w_2^+)^*- w_2^- \right)^2 \,dx.
\end{align}
Coming back to \eqref{w220} and noting that
\begin{align*}
\int_{B} |\nabla w_2^+|^2+ |\nabla w_2^-|^2 \,dx +[w_2^+]_s^2+ [w_2^-]_s^2-2 [w_2^+, w_2^-]_s&=\int_{B} |\nabla w_2|^2 \,dx +[w_2]^2_s \\
&=\int_{B} \left(\sigma_2-V\right)w_2^2 \,dx\\
&=\int_{B} \left(\sigma_2-V\right)\left(w_2^+-w_2^- \right)^2 \,dx,
\end{align*}
we then get that the equality in \eqref{w220} holds true, i.e. $w=(w_2^+)^*- w_2^-$ is an eigenfunction to \eqref{equ1} corresponding to $\sigma_2$. In addition, using \eqref{w210} and the fact that $\mbox{supp} \,\, (w_2^+)^* \subset B_{r_0}$ and $\mbox{supp} \,\, w_2^+ \subset B_{r_0}$, we have that
$$
-\int_{\R^N} V (w_2^+)^2 \,dx=-\int_{B} V (w_2^+)^2 \,dx=-\int_{B} V((w_2^+)^*)^2\,dx=-\int_{\R^N} V((w_2^+)^2)^*\,dx.
$$
It then follows from \cite[Theorem 3.4]{LL} that $w_2^+=(w_2^+)^*$,
i.e. $w_2=w$ in $B$. As a consequence, it holds that
$$
w_2(0)=(w_2^+)^*(0)=\|(w_2^+)^*(0)\|_{\infty}>0.
$$
This is impossible and the proof is complete.
\end{proof}

\begin{lem} \label{schange2}
Let $w_2 \in H^1_{0, r}(B)$ be an eigenfunction to \eqref{equ1} corresponding to $\sigma_2$. Then $w_2$ changes sign only once in the radial direction if and only if it holds that
\begin{align} \label{signw2}
w_2(0) \int_{B} w_2 \,dx<0.
\end{align}
\end{lem}
\begin{proof}
Let us assume that $w_2$ change sign only once in the radial direction. We are going to prove that \eqref{signw2} holds true.
Without restriction, in the spirit of Lemma \ref{schange1}, we may assume that $w_2(0)>0$. Let $w_1 \in H^1_{0, r}(B)$ be the positive eigenfunction to \eqref{equ1} corresponding to the first eigenvalue $\sigma_1$. Using the well-known moving plane method, we know that $w_1$ is strictly decreasing in the radial direction. From the assumption and the fact that $w_2(0)>0$, we know that there exists $0<r_0<1$ be such that $w_2 \geq 0$ in $B_{r_0}$ and $w_2 \leq 0$ in $B \backslash B_{r_0}$. Accordingly, we conclude that
$$
0=\int_{B} w_2w_1 \,dx=\int_{B_{r_0}} w_2w_1 \,dx+\int_{B \backslash B_{r_0}} w_2w_1 \,dx>w_1(r_0) \int_{B} w_2 \,dx.
$$
It then results in
$$
\int_{B} w_2 \,dx<0.
$$
Therefore, \eqref{signw2} holds true.

Let us now assume that \eqref{signw2} holds true. We are going to verify that $w_2$ changes sign only once in the radial direction. For this, we shall argue by contradiction that $w_2$ changes sign at least twice in the radial direction. Without restriction, we may assume that $w_2(0)>0$. Thereby, there exist $0<r_1<r_2<r_3<1$ such that
$$
w_2(r_1)>0, \quad w_2(r_2)<0, \quad w_2(r_3)>0.
$$
Let $W_2$ be the s-harmonic extension of $w_2$. It is clear that $W_2$ radially symmetric with respect to $x$. We are about to prove that $W_2(0, \cdot) \geq 0$ in $\R^+$. Define $\widetilde{W}_2(|x|, t):=W_2(x, t)$. We suppose by contradiction that there exists $t_0>0$ such that $\widetilde{W}_2(0, t_0)<0$. It then yields that $\widetilde{\Omega}_2 \neq \emptyset$, where $\widetilde{\Omega}_2$ is defined by \eqref{defo}. Since $\widetilde{\Omega}_1$ and $\widetilde{\Omega}_2$ are path connected, then there exist $\eta_1 \in C([0, 1], \widetilde{\Omega}_1)$ and $\eta_2 \in C([0, 1], \widetilde{\Omega}_2)$ such that
$$
\eta_1(0)=(r_1, 0), \quad \eta_1(1)=(r_3, 0), \quad \eta_2(0)=(r_2, 0), \quad \eta_2(1)=(0, t_0).
$$
Clearly, according to \cite[Lemma 6.3]{FW}, we see that $\eta_1$ and $\eta_2$ have to intersect in $\R^2_+$. This is impossible. It then follows that $W_2(0, \cdot) \geq 0$. Since $w_2 \in L^1(\R^N)$, by \eqref{defw} and the dominated convergence theorem, then
$$
\lim_{t \to +\infty} t^NW_2(0, t)=\lim_{t \to +\infty} P(N, s) \int_{\R^N} \frac{w_2(y)}{\left(1+ \frac{|y|^2}{t^2}\right)^{\frac{N+2s}{2}}} \,dy=P(N, s) \int_{\R^N} w_2 \,dx.
$$
This implies that
$$
\int_{\R^N} w_2 \,dx \geq 0.
$$
We then reach a contradiction. It then in turn shows that $w_2$ changes sign only once in the radial direction. This completes the proof.
\end{proof}

\begin{lem} \label{schange3}
Let $\lambda_{2, s}$ be the second eigenvalue to the following eigenvalue problem in $H^1_{0,r}(B)$,
\begin{align} \label{equ11}
-\Delta \varphi + (-\Delta)^s \varphi=\lambda \varphi \quad \mbox{in} \,\,\, B, \quad \varphi=0 \quad \mbox{in} \,\,\, \R^N \backslash {B},
\end{align}
where $0<s \leq 1$. Let $\varphi_{2,s}$ be an eigenfunction to \eqref{equ11} corresponding to $\lambda_{2,s}$. Then $\varphi_{2,s}$ changes sign only once in the radial direction and it holds that
$$
\varphi_{2,s}(0)\int_{B} \varphi_{2,s} \,dx<0.
$$
\end{lem}
\begin{proof}
Let us first consider the case that $s=1$. We are going to assert that
\begin{align} \label{s10}
\varphi_{2,1}(0) \int_{B} \varphi_{2,1} \,dx<0.
\end{align}
When $s=1$, it is well-known that $\varphi_{2,1}$ changes sign only once in the radial direction. In addition, by \eqref{equ11} and the divergence theorem, we observe that
\begin{align} \label{s11}
\lambda_{2,1}\int_{B} \varphi_{2,1} \,dx= -2\int_{B} \Delta \varphi_{2,1} \,dx=-2 \int_{\partial B} \partial_{\nu} \varphi_{2,1} \,dS,
\end{align}
where $\partial_{\mu}$ denotes the outer normal derivative on $\partial B$. Without restriction, we shall assume that $\varphi_{2,1}(0)>0$. Since $\varphi_{2,1}$ changes sign only once in the radial direction, by the classical Hopf's Lemma, then $\partial_{\nu} \varphi_{2,1}>0$ on $\partial B$. Noting that $\lambda_{2,1}>0$ and \eqref{s11}, we then derive that
$$
\int_{B} \varphi_{2,1} \,dx<0.
$$
It then leads to \eqref{s10}.

In the sequel, we shall assume that $0<s < 1$. Let $\varphi_{1,s} \in H^1_{0, r}(B)$ and $\|\varphi_{1,s}\|_2=1$ be the positive eigenfunction to \eqref{equ11} corresponding to the first eigenvalue $\lambda_{1,s}$, where
\begin{align*} 
\lambda_{1,s}:=\inf_{\varphi \in H^1_{0, r}(B) \backslash \{0\}} \frac{\displaystyle\int_{B} |\nabla \varphi|^2 \,dx + \int_{\R^N} |(-\Delta)^{\frac s 2} \varphi|^2 \,dx}{\displaystyle \int_{B} |\varphi|^2\,dx}=\inf_{\substack{w \in H^1_{0, r}(B) \backslash \{0\} \\ \|\varphi\|_2=1}}\displaystyle \int_{B} |\nabla \varphi|^2 \,dx + \int_{\R^N} |(-\Delta)^{\frac s 2} \varphi|^2 \,dx.
\end{align*}
Let $\varphi \in H^1_{0, r}(B)$ and $\|\varphi\|_2=1$. Using the above variational characterization of $\lambda_{1,s}$, we then have that
$$
\lambda_{1,s} \leq \int_{B} |\nabla \varphi|^2 \,dx +\int_{\R^N} |(-\Delta)^{\frac s 2} \varphi|^2 \,dx = \|\varphi\|_{H_0^1}^2.
$$
It then follows that $\{\lambda_{1,s}\}$ is bounded in $\R$. Moreover, testing \eqref{equ11} by $\varphi_{1,s}$, we then get that $\{\varphi_{1,s}\}$ is bounded in $H^1_{0,r}(B)$. Further, invoking \cite[Theorem 1.1]{SVWZ}, we then derive that $\|\varphi_{1,s}\|_{\infty} \leq C_0$, where $C_0>0$ is a constant depending on $N$ and $\|\varphi\|_{H_0^1}$.

Let $0<s_0<s<1$ for some $0<s_0<1$. Let $\mathcal{A}$ be a collection of symmetric matrices in $\R^{N \times N}$ whose eigenvalues belong to $[\lambda, \Lambda]$, where $0<\lambda \leq \Lambda$. Let $\mathcal{L}$ be a collection of linear operators. We say that a linear operator $L$ belongs to $\mathcal{L}$ if its corresponding kernel $K$ satisfies that
$$
0 \leq K(x) \leq \frac{1}{|x|^{N+2s}}.
$$
Define the extremal operators by
$$
{M}^+(X):=\max_{A \in \mathcal{A}} \mbox{tr}(AX), \quad {M}^-(X):=\min_{A \in \mathcal{A}} \mbox{tr}(AX),
$$
$$
{M}_{\mathcal{L}}^+(u)(x):=\sup_{L \in \mathcal{L}} Lu(x), \quad {M}_{\mathcal{L}}^+(u)(x):=\inf_{L \in \mathcal{L}} Lu(x),
$$
where
$$
Lu(x):=C(N,s)\int_{\R^N} \left(u(x+y)-u(x)-\nabla u(x) \cdot y \chi_{B} (y)\right) K(y) \,dy. 
$$
It then follows that
$$
{M}^-(D^2 \varphi_{1,s}) + {M}^-_{\mathcal{L}}(\varphi_{1,s}) \leq \Delta \varphi_{s, 1}-(-\Delta)^s  \varphi_{s, 1}=-\lambda_{1,s} \varphi_{s, 1} \leq C_1,
$$
$$
{M}^+(D^2 \varphi_{1,s}) + {M}^+_{\mathcal{L}}(\varphi_{1,s}) \geq \Delta \varphi_{s, 1}-(-\Delta)^s  \varphi_{s, 1}=-\lambda_{1,s} \varphi_{s, 1} \geq -C_1,
$$
where $C_1>0$ is a constant dependent on $s_0$, $N$ and $\|\varphi\|_{H_0^1}$. Following closely the proof of \cite[Theorem 4.1]{CM} and using \cite[Corollary 4.1]{DPV}, we can conclude that there exists $\alpha>0$ and $C>0$ depending on $\lambda$, $\Lambda$, $s_0$, $N$ and $\|\varphi\|_{H_0^1}$ such that
$$
|\varphi_{1,s}(x)-\varphi_{1,s}(0)| \leq C |x|^{\alpha}.
$$
As a consequence, we are able to get that
\begin{align} \label{ibdd}
\|\varphi_{1,s}\|_{C^{0, \alpha}(\overline{B_r})} \leq C' \left(\|\varphi_{s, 1}\|_{\infty} + 1\right) \leq C_2, \quad 0<r<1,
\end{align}
where $C_2>0$ is a constant depending on $\lambda$, $\Lambda$, $s_0$, $r$, $N$ and $\|\varphi\|_{H_0^1}$. Furthermore, utilizing Lemma \cite[Lemma 2.1]{BMS} and \cite[Corollary 4.2]{DPV}, we know that there exists a constant $\widetilde{C}>0$ dependent on $s_0$ and $N$ such that
\begin{align} \label{bbdd}
|\varphi_{1,s}(x)| \leq \widetilde{C} d(x), \quad \forall \,\, x \in B,
\end{align}
where $d(x):=\mbox{dist}(x, \R^N \backslash B)=1-|x|$ for $x \in B$. Define $B^{\delta}:=\left\{x \in B : d(x)<\delta\right\}$. 
Thanks to \eqref{bbdd}, we then get that
\begin{align} \label{bbdd1}
\lim_{\delta \to 0} \sup_{s_0 < s<1} \|\varphi_{1,s}\|_{L^\infty(B^{\delta})}=0.
\end{align}
At this point, taking into account \eqref{ibdd} and \eqref{bbdd1} together with Arzel\`a-Ascoli's theorem, we then can obtain that $\{\varphi_{1,s}\}_{s \in (s_0, 1]}$ is compact in $C(\overline{B})$. Let $0<\widetilde{s} \leq 1$ and $\{s_n\} \subset (0, 1)$ be such that $s_n \to \widetilde{s}$ as $n \to \infty$. Therefore, we are able to conclude that there exists $\widetilde{\varphi} \in C(\overline{B})$ such that $\varphi_{1,s_n} \to \widetilde{\varphi}$ in  $C(\overline{B})$ as $n \to \infty$. Since $\{\lambda_{1,s_n}\}$ is bounded in $\R$, then there exists $\widetilde{\lambda} \in \R$ such that $\lambda_{1,s_n} \to \widetilde{\lambda}$ as $n \to \infty$. 
Since $\{\varphi_{1, s_n}\}$ is bounded in $H^1_{0,r}(B)$, then $\varphi_{1,s_n} \wto \widetilde{\varphi}$ in $H^1_{0,r}(B)$ as $n \to \infty$. In addition, we find that that $(-\Delta)^{s_n} \varphi_{1,s_n} \to (-\Delta)^{\widetilde{s}} \widetilde{\varphi}$ as $n \to \infty$ in the sense of distribution. Since $\varphi_{1,s_n} \in H^1_{0,r}(B)$ solves the equation
$$
-\Delta \varphi_{1,s_n} + (-\Delta)^{s_n} \varphi_{1,s_n}=\lambda_{1,s_n} \varphi_{1,s_n} \quad \mbox{in} \,\,\, B, \quad \varphi_{s_n}=0 \quad \mbox{in} \,\,\, \R^N \backslash {B},
$$ 
then $ \widetilde{\varphi} \in H^1_{0,r}(B)$ satisfies the equation
\begin{align} \label{equ112}
-\Delta\widetilde{\varphi} + (-\Delta)^{\widetilde{s}}\widetilde{\varphi}=\widetilde{\lambda} \widetilde{\varphi} \quad \mbox{in} \,\,\, B, \quad \widetilde{\varphi}=0 \quad \mbox{in} \,\,\, \R^N \backslash {B}.
\end{align}
Noticing that $\varphi_{1,s_n}>0$ in $B$ and $\lambda_{1,s_n}>0$, we then obtain that $\widetilde{\varphi} \geq 0$ in $B$ and $\widetilde{\lambda} \geq 0$. Further, by the strong maximum principle, we then find that $\widetilde{\varphi}>0$ in $B$. It then follows that $\widetilde{\lambda}>0$. If not, one then has that $\widetilde{\varphi}=0$. Therefore, we get that $\widetilde{\varphi}=\varphi_{1, \widetilde{s}}$ and $\widetilde{\lambda}=\lambda_{1,\widetilde{s}}$, namely it holds that $\varphi_{1,s_n} \to \varphi_{1,\widetilde{s}}$ in $C(\overline{B})$ as $n \to \infty$. Accordingly, we are able to conclude that $\varphi_{1,s} \to \varphi_{1,\widetilde{s}}$ in $C(\overline{B})$ as $s \to \widetilde{s}$.

Applying the variational characterization of $\lambda_{2,s}$, we know that
\begin{align} \label{defl2}
\begin{split}
\lambda_{2,s}&=\inf_{\substack{\varphi \in H^1_{0, r}(B) \backslash\{0\} \\ \langle \varphi, \varphi_{1,s} \rangle_2 =0}}\frac{\displaystyle \int_{B} |\nabla \varphi|^2 \,dx + \int_{\R^N} |(-\Delta)^{s/2} \varphi|^2 \,dx}{\displaystyle \int_{B} |\varphi|^2\,dx} \\
&=\inf_{\substack {S \subset  H^1_{0, r}(B) \\ \mbox{dim}(S)=2}} \sup_{\varphi \in S \backslash \{0\}}\frac{\displaystyle \int_{B} |\nabla \varphi|^2 \,dx + \int_{\R^N} |(-\Delta)^{s/2} \varphi|^2 \,dx}{\displaystyle \int_{B} |\varphi|^2\,dx}.
\end{split}
\end{align}
Let $S:=\mbox{span} \{\varphi_{1,1}, \varphi_{2,1}\}$, where $\varphi_{1,1} \in H^1_{0,r}(B)$ and $\varphi_{1,2} \in H^1_{0,r}(B)$ are eigenfunctions to \eqref{equ11} corresponding to the eigenvalues $\lambda_{1,1}$ and $\lambda_{2,1}$ respectively.
Therefore, taking into account \eqref{defl2}, we obtain that
$$
\lambda_{2,s} \leq \sup_{\varphi \in S \backslash \{0\}}\frac{\displaystyle \int_{B} |\nabla \varphi|^2 \,dx + \int_{\R^N} |(-\Delta)^{s/2} \varphi|^2 \,dx}{\displaystyle \int_{B} |\varphi|^2\,dx}.
$$
This implies that $\{\lambda_{2,s}\}$ is bounded in $\R$. Let $\varphi_{2,s} \in H^1_{0, r}(B)$ be an eigenfunction corresponding to $\lambda_{2,s}$ with $\|\varphi_{2,s}\|_2=1$ and $\varphi_{2,s}(0) \neq 0$. In fact, if $\varphi_{2,s}(0)=0$, arguing as the proof of Lemma \ref{schange1}, we can prove that $\varphi_{2,s}=0$ in $B$. Without restriction, we may assume that $\varphi_{2,s}>0$. Observe that $\{\varphi_{2,s}\}$ is bounded in $H^1_{0,r}(B)$, because $\{\lambda_{s,2}\}$ is bounded in $\R$. Similarly, we are able to show that $\varphi_{2,s} \to \varphi_{2,\widetilde{s}}$ in $C(\overline{B})$ as $s \to \widetilde{s}$ as previously. 
As a consequence, making use of \eqref{s10}, we are able to conclude that there exists $s_0>0$ such that
$$
\varphi_{2,s}(0) \int_{B} \varphi_{2,s} \,dx<0, \quad s_0<s<1.
$$
Therefore, reasoning as the proof of Lemma \ref{schange2}, one can show that $\varphi_{2,s}$ changes sign precisely once in the radial direction for any $s_0<s<1$. Define
$$
s_*:=\inf \left\{0<s_0 \leq 1: \varphi_{2,s} \,\,\mbox{changes sign precisely once in the radial direction for any} \,\, s_0<s<1\right\}.
$$
If $s_*=0$, then the proof is completed. Let us assume that $s_*>0$. As a result, we find that $\varphi_{s_*, 1}$ changes sign precisely once in the radial direction. In the same spirit of Lemma \ref{schange2}, it then holds that
$$
\varphi_{s_*,2}(0) \int_{B} \varphi_{s_*,2} \,dx<0.
$$
On the other hand, we know that there exists $\{s_n\} \subset (0, s^*)$ satisfying $s_n \to s_*$ as $n \to \infty$ such that
$$
\varphi_{s_n,2}(0) \int_{B} \varphi_{s_n,2} \,dx \geq 0.
$$
However, this implies that
$$
\varphi_{s_*,2}(0) \int_{B} \varphi_{s_*,2} \,dx \geq 0.
$$
We then reach a contradiction and the desired conclusion follows. This completes the proof.
\end{proof}

\begin{lem} \label{schange0}
Let $w_2 \in H^1_{0, r}(B)$ be an eigenfunction to \eqref{equ1} corresponding to $\sigma_2$. Then $w_2$ changes sign only once in the radial direction.
\end{lem}
\begin{proof}
Let $0 \leq \tau \leq 1$. Define $V_{\tau}:= \tau V$ in $B$, where $V=-(p-1)u^{p-2}$. Let $w_{1, \tau} \in H^1_{0,r}(B)$ be the eigenfunction to the following eigenvalue problem corresponding to the first eigenvalue $\sigma_{1, \tau}$ in $H^1_{0,r}(B)$,
\begin{align} \label{equ22}
-\Delta w + (-\Delta)^s w + V_{\tau} w=\sigma w \quad \mbox{in} \,\,\, B, \quad w=0 \quad \mbox{in} \,\,\, \R^N \backslash {B}.
\end{align}
Let $\{\tau_n\} \subset (0, 1]$ be such that $\tau_n \to \tau$ as $n \to \infty$. Let $w_{1, \tau}$ and $w_{1, \tau_n}$ be the eigenfunctions to \eqref{equ22} corresponding to the eigenvalues $\sigma_{1, \tau}$ and $\sigma_{1, \tau_n}$ respectively. Observe that
\begin{align*}
\sigma_{1, \tau_n} &= \inf_{w \in H^1_{0,r}(B) \backslash \{0\}}\frac{\displaystyle \int_{B} |\nabla w|^2 \,dx + \int_{\R^N} |(-\Delta)^{s/2} w|^2 \,dx + \int_{B} V_{\tau_n} |w|^2 \,dx}{\displaystyle \int_{B} |w|^2\,dx} \\
& \leq \sigma_{1, \tau} +  \frac{\displaystyle \int_{B} \left(V_{\tau_n} -V_{\tau}\right) |w_{1, \tau}|^2 \,dx}{\displaystyle \int_{B} |w_{1, \tau}|^2\,dx} 
=\sigma_{1, \tau} + o_n(1).
\end{align*}
Similarly, one can derive that 
\begin{align*}
\sigma_{1, \tau}  &= \inf_{w \in H^1_{0,r}(B) \backslash \{0\}}\frac{\displaystyle \int_{B} |\nabla w|^2 \,dx + \int_{\R^N} |(-\Delta)^{s/2} w|^2 \,dx + \int_{B} V_{\tau} |w|^2 \,dx}{\displaystyle \int_{B} |w|^2\,dx} \\
& \leq \sigma_{1, \tau_n} +  \frac{\displaystyle \int_{B} \left(V_{\tau} -V_{\tau_n}\right) |w_{1, \tau_n}|^2 \,dx}{\displaystyle \int_{B} |w_{1, \tau_n}|^2\,dx} =\sigma_{1, \tau_n} + o_n(1).
\end{align*}
Consequently, we have that $\sigma_{1, \tau_n} \to \sigma_{1, \tau}$ as $n \to \infty$. It then infers that $\{w_{1, \tau_n}\}$ is bounded in $H^1_{0,r}(B)$. Let $\widetilde{w} \in H^1_{0,r}(B)$ be such that $w_{1, \tau_n} \wto \widetilde{w}$ in $H^1_{0,r}(B)$ as $n \to \infty$. This implies that $\widetilde{w}$ is an eigenfunction to \eqref{equ22} corresponding to the eigenvalue $\sigma_{1, \tau}$ in $H^1_{0,r}(B)$. Since $\sigma_{1, \tau}$ is simple by Lemma \ref{simple}, then $\widetilde{w}=\kappa w_{1, \tau}$ for some $\kappa \neq 0$. Without restriction, we may assume $\kappa=1$. As a result, we can conclude that $w_{1, \tau_n} \to w_{1, \tau}$ in $H^1_{0,r}(B)$ as $n \to \infty$.
Since $w_{1, \tau_n} \in H^1_{0, r}(B)$ solves the equation
$$
-\Delta w_{1, \tau_n}  + (-\Delta)^s w_{1, \tau_n}  + V_{\tau_n} w_{1, \tau_n} =\sigma_{1, \tau_n} w_{1, \tau_n}  \quad \mbox{in} \,\,\, B, \quad w_{1, \tau_n} =0 \quad \mbox{in} \,\,\, \R^N \backslash {B},
$$
by \cite[Theorem 1.1]{SVWZ}, then $\|w_{1, \tau_n}\|_{\infty} \leq C_0$, where $C_0$ is a constant depending on $N$ and $s$.
Observe that
$$
\left\{
\begin{aligned}
&-\Delta \left(w_{1, \tau_n}-w_{1, \tau}\right) + (-\Delta)^s \left(w_{1, \tau_n}-w_{1, \tau}\right)   =\left(V_{\tau} w_{1, \tau}-V_{\tau_n} w_{1, \tau_n} \right)+ \left(\sigma_{1, \tau_n} w_{1, \tau_n}- \sigma_{1, \tau} w_{1, \tau} \right)\quad \mbox{in} \,\,\, B, \\
& w_{1, \tau_n}-w_{1, \tau}=0 \,\,\, \mbox{in} \,\,\, \R^N \backslash B.
\end{aligned}
\right.
$$
In view of \cite[Theorem 1.1]{SVWZ}, we also have that $\|w_{1, \tau_n}-w_{1, \tau}\|_{\infty} \leq C_1$, where $C_1>0$ is a constant depending on $s$ and $N$. Further, by \cite[Theorem 1.3]{SVWZ}, we get that $\|w_{1, \tau_n}-w_{1, \tau}\|_{C^{1, \alpha}(\overline{B})} \leq C_2$, where $C_2>0$ is a constant depending on $s$ and $N$. As a result of Arzel\`a-Ascoli's theorem, we then derive that $w_{1, \tau_n} \to w_{1, \tau}$ in $C(\overline{B})$ as $n \to \infty$. This then implies that $w_{1, \tau} \to w_{1, \widetilde{\tau}}$ in $C(\overline{B})$ as $\tau \to \widetilde{\tau}$ for any $0 \leq \widetilde{\tau} \leq 1$.

Let $w_{2, \tau} \in H^1_{0,r}(B)$ be an eigenfunction to \eqref{equ22} corresponding to the second eigenvalue $\sigma_{2, \tau}$ in $H^1_{0,r}(B)$.
Without restriction, we may assume that $w_{2, \tau}(0)>0$ and $\|w_{2, \tau}\|_2=1$. On the contrary, if $w_{2, \tau}(0)=0$, by Lemma \ref{schange1}, then $w_{2, \tau}=0$ in $B$.
Reasoning as before, we are also able to conclude that $w_{2, \tau} \to w_{2, \widetilde{\tau}}$ in $C(\overline{B})$ as $\tau \to \widetilde{\tau}$. By Lemma \ref{schange3}, we know that
$$
w_{2,0}(0) \int_{B} w_{2, 0} \,dx<0.
$$
Therefore, by Lemma \ref{schange2}, there exists $\tau_0>0$ such that  $w_{2, \tau}$ changes sign only once in the radial direction for any $0 \leq \tau \leq \tau_0$. Define
$$
\tau_*:= \sup \left\{ 0 \leq \tau_0 \leq 1 : w_{2, \tau} \,\, \mbox{changes sign only once in the radial direction for any} \,\, 0 \leq \tau \leq \tau_0\right\}.
$$
It holds that $\tau_* \leq 1$. If $\tau_*=1$, then the proof is completed. If not, then we assume that $\tau_*<1$. Therefore, we see that $w_{2, \tau_*}$ changes sign only once in the radial direction. In view of Lemma \ref{schange2}, we then get that
$$
w_{2,\tau_*} (0)\int_{B} w_{2, \tau_*} \,dx<0.
$$
On the other hand, we know that there exists $\{\tau_n\} \subset (\tau_*, 1)$ such that $\tau_n \to \tau_*$ as $n \to \infty$ and
$$
w_{2,\tau_n} (0)\int_{B} w_{2, \tau_n} \,dx \geq 0.
$$
This clearly leads to
$$
w_{2,\tau_*} (0)\int_{B} w_{2, \tau_*} \,dx \geq 0.
$$
We then reach a contradiction. It then implies that $\tau_*=1$ and $w_{2,1}$ changes sign only once in the radial direction. Thus the proof is complete.
\end{proof}

\begin{lem} \label{mp}
Let $h \in H^1_0(B)$. Suppose that there exists $0<r_0<1$ such that, for any $r_0<r<1$,
\begin{align} \label{mh}
\int_{B \backslash A_r} \frac{h(y)}{|x-y|^{N+2s}} \,dy \geq 0, \quad \forall \,\, x \in A_r,
\end{align}
where $A_r:=\left\{ x \in B : r<|x|<1 \right\}$. Let $c \in L^{\infty}(B)$ and $r_0<r<1$ be such that $\lambda_{s,1}(A_r)>\|c\|_{\infty}$, where $\lambda_{1,s}(A_r)>0$ is the first eigenvalue of the operator $-\Delta + (-\Delta)^s$  in $A_r$ with the Dirichlet condition defined by 
\begin{align} \label{equar}
 \lambda_{1,s}(A_r):&=\inf_{w \in H^1_{0, r}(A_r) \backslash \{0\}} \frac{\displaystyle\int_{A_r} |\nabla w|^2 \,dx + \int_{\R^N} |(-\Delta)^{\frac s 2} w|^2 \,dx}{\displaystyle \int_{A_r} |w|^2\,dx}.
\end{align}
Let $w \in H^1_0(B)$ be a solution to the problem
\begin{align} \label{equ333}
\left\{
\begin{aligned}
&-\Delta w + \left(-\Delta\right)^s w \geq c(x) w \quad &\mbox{in} \,\,\, A_r, \\
&w=h \quad & \mbox{in} \,\,\, \R^{N} \backslash A_r,
\end{aligned}
\right.
\end{align}
If $w^-\mid_{A_r} \in H^1_0(A_r)$, then $w \geq 0$ in $A_r$. 
\end{lem}
\begin{proof}
Let $w=w^+-w^-$. Testing \eqref{equ333} by $w^-\mid_{A_r}$ and using \eqref{equar}, we get that
\begin{align*}
-\int_{A_r} c(x) (w^-)^2 \,dx &\leq -\int_{A_r} |\nabla w^-|^2 \,dx + \frac{C(N,s)}{2}\int_{\R^N} \int_{\R^N} \frac{\left(w(x)-w(y)\right)\left(w^-\mid_{A_r}(x)-w^-\mid_{A_r}(y)\right)}{|x-y|^{N+2s}}\,dxdy \\
& \leq  -\int_{A_r} |\nabla w^-|^2 \,dx - \frac{C(N,s)}{2} \int_{\R^N} \int_{\R^N} \frac{\left(w^-\mid_{A_r}(x)-w^-\mid_{A_r}(y)\right)^2}{|x-y|^{N+2s}}\,dxdy \\
& \quad -C(N,s)\int_{A_r} \int_{A_r} \frac{w^+(x)w^-(y)}{|x-y|^{N+2s}}\,dxdy-C(N,s)  \int_{A_r} \int_{\R^N \backslash A_r} \frac{w^-(x)w(y)}{|x-y|^{N+2s}}\,dxdy\\
& \leq -\lambda_{1,s}(A_r) \int_{A_r} (w^-)^2 \,dx-C(N,s) \int_{A_r} \int_{\R^N \backslash A_r} \frac{w^{-}(x)h(y)}{|x-y|^{N+2s}}\,dxdy.
\end{align*}
Since $\lambda_{1,s}(A_r)> \|c\|_{\infty}$, by \eqref{mh}, then $w^-=0$ in $A_r$. This completes the proof.
\end{proof}

\begin{lem} \label{hopf}
Let $w_2 \in H^1_{0, r}(B)$ be an eigenfunction to \eqref{equ1} corresponding to $\sigma_2$. Then it holds that
$$
w_2(0)\liminf_{|x| \nearrow 1} \frac{w_2(x)}{1-|x|}<0.
$$
\end{lem}

\begin{proof}
First we claim that
\begin{align} \label{w2}
w_2(0)\int_B \frac{w_2(y)}{|x-y|^{N+2s}} \,dy<0, \quad \forall\,\, x \in \partial B.
\end{align}
Indeed, by Lemma \ref{schange0}, we know that $w_2$ changes sign precisely once in the radial direction. 
Without restriction, we may assume that $w_2(0)>0$ and there exists $0<r_0<1$ such that $w_2(r) \geq 0$ for $0<r<r_0$ and $w_2(r) \leq 0$ for $r_0<r<1$. Since $w_2$ is radially symmetric, then
\begin{align*}
\int_B \frac{w_2(y)}{|x-y|^{N+2s}} \,dy&=\frac{1}{\omega_{N-1}} \int_{\partial B} \int_B  \frac{w_2(y)}{|x-y|^{N+2s}} \,dy dS \\
&=\frac{1}{\omega_{N-1}}\int_B  \int_{\partial B} \frac{w_2(y)}{|x-y|^{N+2s}} \,dS dy=\frac{1}{\omega_{N-1}} \int_B w_2 g dy,
\end{align*}
where 
$$
g(y):=\int_{\partial B} \frac{1}{|x-y|^{N+2s}} \,dS.
$$
It follows from the proof of \cite[Proposition 4.5]{DP} that $g$ is radially symmetric and increasing in the radial direction. Moreover, it satisfies that $g(0)>0$ and $\lim_{|y| \to 1^-} g(y)=+\infty$. Let $w_1 \in H^1_{0,r}(B)$ be the eigenfunction to \eqref{equ1} corresponding to the first eigenvalue $\sigma_1$. It holds that $w_1>0$ in $B$ and $\langle w_1, w_2 \rangle_2=0$. In addition, we know that $w_1$ is strictly decreasing in the radial direction. Define
$$
\widetilde{g}:=g-\frac{g(r_0)}{w_1(r_0)} w_1.
$$
It follows that $\widetilde{g} \leq 0$ for $0<r<r_0$ and $\widetilde{g} \geq 0$ for $r_0<r<1$. Therefore, we conclude that
\begin{align*}
\int_{B} w_2 g \,dy=\int_{B} w_2 \widetilde{g} \,dy=\int_{B}\left(w_2^+-w_2^-\right) \widetilde{g} \,dy=\int_{B_{r_0}}w_2^+ \widetilde{g} \,dy-\int_{B\backslash B_{r_0}} w_2^-\widetilde{g} \,dy<0.
\end{align*}
This implies that \eqref{w2} holds true. 

In the sequel, without restriction, we shall always assume that $w_2(0)>0$. Therefore, by \eqref{w2}, there exists $r_0<r_1<1$ such that, for any $r_1<r<1$,
\begin{align} \label{w222}
\int_{B \backslash A_r} \frac{w_2(y)}{|x-y|^{N+2s}} \,dy<0, \quad \forall \,\, x \in \partial B,
\end{align}
where $A_{r}:=\left\{ x \in B : r<|x|<1 \right\}$. It follows that $w_2 \leq 0$ in $A_{r_1}$, because $w_2(r) \leq 0$ for $r_0<r<1$. Define $c(x):=\sigma_2-V(x)$ for $x \in B$. Observe that $\lim_{r \to 1^-}\lambda_{1,s}(A_r)=+\infty$ by \cite[Corollary 1.2]{KT}. Therefore, we are able to choose $0<r_1<1$ closer to $1$ if necessary such that $\lambda_{1,s}(A_r)> \|c\|_{\infty}$ for any $r_1<r<1$. Let $\varphi \in H^1_0(A_r)$ be the positive solution to the equation
$$
-\Delta \varphi+ (-\Delta)^s \varphi=1 \quad \mbox{in} \,\,\, A_r,  \quad \varphi=0 \,\,\, \mbox{in} \,\,\, \R^N \backslash A_r.
$$ 
Let $0<\widetilde{r}<1$ be such that $B_{\widetilde{r}} \cap A_r=\emptyset$. Let $\widetilde{\varphi} \in C^{\infty}_0(B_{\widetilde{r}})$ be non-negative such that $\widetilde{\varphi} \geq 1$ in $B_{\widetilde{r}/2}$. Define $\psi:=\varphi + \alpha \widetilde{\varphi}$, where $\alpha>0$ is a constant to be determined later. Let $\phi \in H^1_0(A_r)$ be non-negative. Consequently, we conclude that
\begin{align} \label{wmp}
\begin{split}
&\int_{A_r} \nabla \psi \cdot \nabla \phi \,dx+\frac{C(N,s)}{2}\int_{\R^N} \int_{\R^N} \frac{\left(\psi(x)-\psi(y)\right)\left(\phi(x)-\phi(y)\right)}{|x-y|^{N+2s}}\,dxdy \\
&=\int_{A_r} \nabla \varphi \cdot \nabla \phi \,dx+\frac{C(N,s)}{2}\int_{\R^N} \int_{\R^N} \frac{\left(\varphi(x)-\varphi(y)\right)\left(\phi(x)-\phi(y)\right)}{|x-y|^{N+2s}}\,dxdy \\
& \quad + \frac{\alpha C(N,s)}{2} \int_{\R^N} \int_{\R^N} \frac{\left(\widetilde{\varphi}(x)-\widetilde{\varphi}(y)\right)\left(\phi(x)-\phi(y)\right)}{|x-y|^{N+2s}}\,dxdy \\
&= \int_{A_r} \phi \,dx + \frac{\alpha C(N,s)}{2}  \int_{\R^N} \int_{\R^N} \frac{\left(\widetilde{\varphi}(x)-\widetilde{\varphi}(y)\right)\left(\phi(x)-\phi(y)\right)}{|x-y|^{N+2s}}\,dxdy \\
&=\int_{A_r} \phi \,dx - \alpha C(N,s) \int_{A_r} \int_{B_{\widetilde{r}}} \frac{\phi(x)\widetilde{\varphi}(y)}{|x-y|^{N+2s}}\,dydx \\
& \leq \int_{A_r} \phi \,dx -  \alpha C(N,s)  \int_{A_r} \int_{B_{\widetilde{r}/2}} \frac{\phi(x)}{|x-y|^{N+2s}}\,dydx \\
& \leq \left(1-\alpha C\right)  \int_{A_r} \phi \,dx,
\end{split}
\end{align}
where we used the fact that $|x-y|>\widetilde{r}/2$ for any $x \in A_r $ and $y \in B_{\widetilde{r}/2}$ for the last inequality and $C=C(N,s, \widetilde{r})>0$.
At this point, we are able to take $\alpha>0$ large enough such that $1-\alpha C \leq -\|c\|_{\infty} \psi$, because of $\psi \geq 0$ and $\psi \in L^{\infty}(A_r)$. It then follows from \eqref{wmp} that
\begin{align} \label{equ1111}
-\Delta \psi + (-\Delta)^s \psi \leq - \|c\|_{\infty} \psi \leq c(x) \psi \quad \mbox{in} \,\,\, A_r,  \quad \psi \geq 0 \,\,\, \mbox{in} \,\,\R^N \backslash A_r.
\end{align}
Furthermore, in view of \eqref{w222}, we find that there exists $\eps>0$ small enough such that
$$
\int_{B \backslash A_r} \frac{w_2(y)+\eps \psi(y)}{|x-y|^{N+2s}} \,dy<0, \quad \forall \,\, x \in A_r.
$$
Combining \eqref{equ1} and \eqref{equ1111}, we conclude that
$$
-\Delta (w_2+\eps \psi)+ (-\Delta)^s  (w_2+\eps \psi) \leq  c(x)   (w_2+\eps \psi) \quad \mbox{in} \,\,\, A_r.
$$
It then follows from Lemma \ref{mp} that $w_2 \leq -\eps \psi=-\eps \varphi$ in $A_r$. As a consequence, by Hopf's lemma (see \cite[Theorem 2.2]{BM}), it holds that
$$
\liminf_{|x| \nearrow 1} \frac{w_2(x)}{1-|x|} \leq \eps \liminf_{|x| \nearrow 1} \frac{-\varphi(x)}{1-|x|}<0.
$$
Thereby, we have the desired conclusion and the proof is complete.
\end{proof}

\begin{proof}[Proof of Theorem \ref{thm1}]
Combining Lemma \ref{schange1}, Lemma \ref{schange0} and Lemma \ref{hopf} together with the fact that $w_2^+=(w_2^+)^*$ from the proof of Lemmas \ref{schange1}, we then derive the desired conclusions. This completes the proof.
\end{proof}

\begin{lem} \label{sign}
Let $\sigma_2$ be the second eigenvalue to \eqref{equ1} in $H^1_{0,r}(B)$. Then it holds that $\sigma_2 \neq -\lambda$. 
\end{lem}
\begin{proof}
We shall argue by contradiction that $\sigma_2=-\lambda$. This then indicates that there exists $w \in H^1_{0, r}(B)\backslash \{0\}$ such that
\begin{align} \label{equ110}
-\Delta w+ (-\Delta)^s w-(p-1)u^{p-2}w=-\lambda w \quad \mbox{in} \,\,\, B, \quad w=0 \quad \mbox{in} \,\, \,\R^N \backslash B.
\end{align}
By Lemma \ref{schange1}, we know that $w(0) \neq 0$. Without restriction, we may assume that $w(0)>0$. Moreover, in view of Lemma \ref{hopf}, we have that $\psi_{w}(1)<0$, where
$$
\psi_w(1):=\liminf_{|x| \nearrow 1} \frac{w(x)}{1-|x|}.
$$ 
Let $\chi \in C^{\infty}_0(\R)$ be such that $0 \leq \chi \leq 1$ in $\R$, $\chi=1$ in $[-1, 1]$ and $\chi=0$ in $ \R \backslash (-2, 2)$. 
Define $\zeta_k(x):=1-\chi(k(1-|x|))$ for $x \in \R^N$. It then follows from \eqref{equ110} that 
\begin{align*}
-\Delta(\zeta_k w) + (-\Delta)^s(\zeta_k w)&=\zeta_k \left(-\Delta w +  (-\Delta)^s w\right) + w \left(-\Delta \zeta_k +  (-\Delta)^s \zeta_k \right) -2 \nabla w \cdot \nabla \zeta_k-I(w, \zeta_k) \\
&=\zeta_k \left(-\lambda w + (p-1) u^{p-2} w\right) +w\left(-\Delta \zeta_k +  (-\Delta)^s \zeta_k \right) -2 \nabla w \cdot \nabla \zeta_k-I(w, \zeta_k),
\end{align*}
where
$$
I(w, \zeta_k):=\frac{C(N, s)}{2} \int_{\R^N} \frac{(w(x)-w(y))(\zeta_k(x)-\zeta_k(y))}{|x-y|^{N+2s}}\,dy.
$$
Invoking the divergence theorem, we know that
$$
\int_{\R^N \backslash \overline{B}}-\Delta(\zeta_k w) \,dx=0.
$$
As a consequence, we have that
\begin{align} \label{w0}
\int_{\R^N}-\Delta(\zeta_k w) + (-\Delta)^s(\zeta_k w) \,dx \nonumber
&=\int_{B}-\Delta(\zeta_k w) + (-\Delta)^s(\zeta_k w) \,dx + \int_{\R^N \backslash  \overline{B}}-\Delta(\zeta_k w) + (-\Delta)^s(\zeta_k w) \,dx \\ \nonumber
&=\int_{B} \zeta_k \left(-\lambda w + (p-1) u^{p-2} w\right) \,dx + \int_{B} w \left(-\Delta \zeta_k \right) -2 \nabla w \cdot \nabla \zeta_k\,dx \\
& \quad + \int_{B}w(-\Delta)^s \zeta_k-I(w, \zeta_k) \,dx + \int_{\R^N \backslash  \overline{B}}(-\Delta)^s(\zeta_k w) \,dx.
\end{align}
It is simple to find that
\begin{align} \label{w01}
\int_{\R^N}-\Delta(\zeta_k w) + (-\Delta)^s(\zeta_k w) \,dx=0.
\end{align}
In what follows, we shall deal with the terms in the right hand side of \eqref{w0}. To begin with, since $0 \leq \zeta_k \leq 1$ and $\lim_{k \to \infty} \zeta_k(x)=1$ for any $x \in B$, by the dominated convergence theorem, then we have that
\begin{align} \label{w02}
\lim_{k \to \infty}\int_{B} \zeta_k \left(-\lambda w + (p-1) u^{p-2} w\right) \,dx=\int_{B} -\lambda w + (p-1) u^{p-2} w \,dx.
\end{align}
Let us next treat the last term in the right hand side of \eqref{w0}. Since $w \in L^{\infty}(B)$, by \cite[Lemma 2.1]{BM}, then $|w(x)| \leq C (1-|x|)$ for any $x \in B$. It then leads to
\begin{align*} 
|w(x)| \leq C |x-y|, \quad \forall \,\, x \in B, y \in \R^N \backslash \overline{B}.
\end{align*}
Therefore, we know that
\begin{align} \label{cw2}
|w(x)| \leq C |x-y| \leq 3^{1-s} C |x-y|^s, \quad \forall \,\, x \in B, y \in B_2 \backslash \overline{B}.
\end{align}
Observe that, for any $x \in \R^N \backslash \overline{B}$,
\begin{align} \label{w11}
\left|(-\Delta)^s(\zeta_k w)(x) \right| &=C(N, s) \left|\int_{\R^N} \frac{\zeta_k(y) w(y)}{|x-y|^{N+2s}} \,dy\right| =C(N, s) \left| \int_{B} \frac{\zeta_k(y) w(y)}{|x-y|^{N+2s}} \,dy \right|,
\end{align}
where we used the fact that $w=0$ in $\R^N \backslash B$.
Noting that $\|\zeta_k w\|_{\infty} \leq \|w\|_{\infty}$ and utilizing \eqref{cw2}, we then have that
\begin{align} \label{w12}
\left|(-\Delta)^s(\zeta_k w)(x) \right|  \leq \widetilde{C} \int_{B} \frac{1}{|x-y|^{N+s}} \,dy, \quad \forall \,\, x \in B_2 \backslash \overline{B},
\end{align}
\begin{align} \label{w13}
\left|(-\Delta)^s(\zeta_k w)(x) \right|  \leq \widetilde{C} \int_{B} \frac{1}{|x-y|^{N+2s}} \,dy, \quad \forall \,\, x \in \R^N \backslash B_2,
\end{align}
where $\widetilde{C}=\widetilde{C}(N, s, \lambda)>0$.
Let us first estimate \eqref{w12}. Notice that, for any $x \in B_2 \backslash \overline{B}$,
\begin{align*}
\int_{B} \frac{1}{|x-y|^{N+s}} \,dy&=\int_0^1 \int_{\partial B} \frac{\rho^{N-1}}{|x-\rho y|^{N+s}} \,dS d\rho
=\int_0^1 \int_{\partial B} \frac{\rho^{N-1}}{||x|e_1-\rho y|^{N+s}} \,dS d\rho \\
&=\int_0^1 \int_{\partial B} \frac{\rho^{N-1}}{||x|^2+\rho^2-2|x| \rho y_1|^{\frac{N+s}{2}}} \,dS d\rho \\
&=\frac{2 \pi^{\frac{N-1}{2}}}{\Gamma\left(\frac{N-1}{2}\right)}\int_0^1 \int_0^{\pi} \frac{\rho^{N-1} \sin^{N-2}\theta}{\left||x|^2+\rho^2-2|x| \rho \cos \theta\right|^{\frac{N+s}{2}}} \,d\theta d\rho \\
&=\frac{2 \pi^{\frac{N-1}{2}}}{\Gamma\left(\frac{N-1}{2}\right)|x|^{N+s}} \int_0^1 {}_2 F_1 \left(\frac{N+s}{2}, \frac s 2 +1; \frac{N}{2}; \frac{\rho^2}{|x|^2}\right) \rho^{N-1}\,d\rho,
\end{align*}
where ${}_2 F_1$ is the hypergeometric function (see \cite[Section 2.5.1]{MOS}) defined by
$$
{}_2 F_1(a, b; c; z)=\frac{\Gamma(c)}{\Gamma(a) \Gamma(b)}  \sum_{k=0}^{\infty} \frac{\Gamma(a+k) \Gamma(b+k)}{\Gamma(c+k)} \frac{z^k}{k!}.
$$
From \cite[Chapter 15]{OD}, we know that
\begin{align*}
{}_2 F_1 \left(\frac{N+s}{2}, \frac s 2+1; \frac{N}{2}; \frac{\rho^2}{|x|^2}\right) 
 \sim \left(1-\frac{\rho^2}{|x|^2}\right)^{-1-s} \frac{\Gamma\left(\frac N 2\right)\Gamma(1+s)}{\Gamma\left(\frac{N+s}{2}\right)\Gamma\left(1+\frac s 2\right)} \quad  \mbox{as} \,\,\,\frac{\rho}{|x|} \to 1^-.
\end{align*}
It then follows that 
$$
{}_2 F_1 \left(\frac{N+s}{2}, \frac s 2+1; \frac{N}{2}; \frac{\rho^2}{|x|^2}\right) \leq C \left(1-\frac{\rho^2}{|x|^2}\right)^{-1-s}, \quad \forall \,\, x \in B_2 \backslash \overline{B}.
$$
As a consequence, it holds that
\begin{align} \label{w14}
\begin{split}
\int_{B} \frac{1}{|x-y|^{N+s}} \,dy &\leq C \int_0^1  \left(1-\frac{\rho^2}{|x|^2}\right)^{-1-s} \rho^{N-1} \,d\rho \leq C \int_0^1  \left(1-\frac{\rho}{|x|}\right)^{-1-s} \,d\rho \\
&=\frac{C|x|}{s} \left(\left(1-\frac{1}{|x|}\right)^{-s} -1\right)\leq \frac{2^{1+s}C} {s(|x|-1)^{s}}, \quad \forall \,\, x \in B_2 \backslash \overline{B}.
\end{split}
\end{align}
In addition, since $|x-y| \geq |x|-1 \geq \frac 13 (|x|+1)$ for $x \in \R^N \backslash B_2$ and $y \in B$, then 
\begin{align}  \label{w15}
\int_{B} \frac{1}{|x-y|^{N+2s}} \,dy \leq \frac{C} {(1+|x|)^{N+2s}}, \quad \forall \,\, x \in \R^N \backslash B_2.
\end{align}
Coming back to \eqref{w11} and using \eqref{w12}, \eqref{w13}, \eqref{w14} and \eqref{w15}, we then get that
\begin{align} \label{w16}
\left|(-\Delta)^s(\zeta_k w)(x) \right| \leq C \left(\frac{1}{(|x|-1)^{s}} \chi_{B_2 \backslash \overline{B}}(x) + \frac{1}{(1+|x|)^{N+2s}}\chi_{\R^N \backslash B_2}(x)\right), \quad \forall \,\, x \in \R^N \backslash B.
\end{align}
Since $w(x)=0$ for any $x \in \R^N \backslash B$, $0 \leq \zeta_k \leq 1$ and $\lim_{k \to \infty} \zeta_k(x)=1$ for any $x \in B$, by \eqref{w11}, \eqref{w16} and the dominated convergence theorem, then
$$
\lim_{k \to \infty}(-\Delta)^s(\zeta_k w)(x)=(-\Delta)^s w(x), \quad \forall \,\, x \in \R^N \backslash \overline{B},
$$
\begin{align} \label{w05}
\lim_{k \to \infty}\int_{\R^N \backslash \overline{B}} (-\Delta)^s(\zeta_k w) \,dx=\int_{\R^N \backslash \overline{B}} (-\Delta)^s w \,dx.
\end{align}

Next, we are going to estimate the remaining terms in the right hand side of \eqref{w0}. Let $0<\eps<1$. We now write
\begin{align} \label{intb}
\int_{B}w(-\Delta)^s \zeta_k-I(w, \zeta_k) \,dx=\int_{B_{1-\eps}} w(-\Delta)^s \zeta_k-I(w, \zeta_k) \,dx + \int_{B \backslash B_{1-\eps}} w(-\Delta)^s \zeta_k-I(w, \zeta_k) \,dx.
\end{align}
Let $k \in \N^+$ be large enough. It then follows that
\begin{align} \label{w001}
(-\Delta)^s \zeta_k(x)=C(N, s) \int_{\R^N \backslash K} \frac{1- \zeta_k(y)}{|x-y|^{N+2s}} \,dy, \quad \forall \,\, x \in B_{1-\eps},
\end{align}
where $K \subset B$ is a compact neighborhood of $B_{1-\eps}$. Observe that
$$
\frac{|1- \zeta_k(y)|}{|x-y|^{N+2s}} \leq \frac{C}{1+|y|^{N+2s}}, \quad \forall \,\, x\in B_{1-\eps}, y \in \R^N \backslash K.
$$
This then implies that $\|(-\Delta)^s \zeta_k\|_{L^{\infty}(B_{1-\eps})} \leq C$. Invoking \eqref{w001} and the dominated convergence theorem, we then get that
$$
\lim_{k \to \infty} (-\Delta)^s \zeta_k(x)=0, \quad \forall \,\, x\in B_{1-\eps}.
$$
Further, by the dominated convergence theorem, we are able to derive that
\begin{align} \label{w21}
\lim_{k \to \infty} \int_{B_{1-\eps}} w(-\Delta)^s \zeta_k \,dx=0.
\end{align}
Similarly, we are able to show that
\begin{align} \label{w22}
\lim_{k \to \infty} \int_{B_{1-\eps}} I(w, \zeta_k) \,dx=0.
\end{align}
We now handle the second term in the right hand side of \eqref{intb}. Since the functions $w$, $(-\Delta)^s \zeta_k$ and $I(w, \zeta_k) $ are radially symmetric, then
\begin{align} \label{w100}
\begin{split}
&\int_{B \backslash B_{1-\eps}} w (-\Delta)^s \zeta_k-I(w, \zeta_k) \,dx \\
&= \omega_{N-1}\int_0^{\eps} \left(w(1-\rho)(-\Delta)^s \zeta_k(1-\rho)-I(w, \zeta_k)(1-\rho) \right) (1-\rho)^{N-1} \,d\rho \\
&= \frac {\omega_{N-1}} {k} \int_0^{\eps k} \left(w(1-k^{-1}\rho)(-\Delta)^s \zeta_k(1-k^{-1}\rho)-I(w, \zeta_k)(1-k^{-1}\rho) \right) (1-k^{-1}\rho)^{N-1} \,d\rho,
\end{split}
\end{align}
where $\omega_{N-1}$ denotes the surface of the unit sphere in $\R^N$.
Define
$$
G(x, \rho):=\left(w(1-k^{-1}\rho)(-\Delta)^s \zeta_k(1-k^{-1}\rho)-I(w, \zeta_k)(1-k^{-1}\rho) \right) (1-k^{-1}\rho)^{N-1}.
$$
Since $|w(x)| \leq C(1-|x|)$ for any $x \in B$ by \cite[Lemma 2.1]{BM}, then 
$$
k^s |w(1-k^{-1}\rho)| \leq C k^{s-1} \rho \leq C \rho, \quad \forall \,\, 0 <\rho<k \eps.
$$
In view of \cite[Proposition 6.3 $\&$ Lemma 6.8]{DFW}, we then get that
\begin{align} \label{intg}
\frac{|G(x, \rho)|}{k} \leq C \left(\frac{k^{s-1} \rho} {1+\rho^{1+2s}} + \frac{k^{s-1}} {1+\rho^{1+s}}\right) \leq C\left(\frac{\rho^{s}} {1+\rho^{1+2s}} + \frac{\rho^{s-1}} {1+\rho^{1+s}}\right), \quad \forall \,\, 0 <\rho<k \eps,
\end{align}
which leads to
$$
\lim_{k \to \infty}\frac{G(x, \rho)}{k}=0.
$$
As a result of \eqref{w100}, \eqref{intg} and the dominated convergence theorem, we then have that
\begin{align*}
\lim_{k \to \infty }\int_{B \backslash B_{1-\eps}} \left|w (-\Delta)^s \zeta_k-I(w, \zeta_k) \right| \,dx=0.
\end{align*}
This jointly with \eqref{w21} and \eqref{w22} infers that
\begin{align} \label{w03}
\lim_{k \to \infty}\int_{B}w(-\Delta)^s \zeta_k-I(w, \zeta_k) \,dx=0.
\end{align}

Let us now estimate the last remaining term in the right hand side of \eqref{w0}. It is simple to compute that
\begin{align} \label{deffd}
-\nabla \zeta_k(x)=\nabla (\chi(k(1-|x|)))=-k \chi'(k(1-|x|)) \frac{x}{|x|},
\end{align}
\begin{align} \label{defsd}
-\Delta \zeta_k(x)=\Delta (\chi(k(1-|x|)))=k^2 \chi''(k(1-|x|))-\frac{k(N-1)}{|x|}\chi'(k(1-|x|)).
\end{align}
Let $k \in \N^+$ be large enough. Therefore, we see that $\nabla \zeta_k(x)=0$ and $-\Delta \zeta_k(x)=0$ for any $x \in B_{1-\eps}$. This immediately leads to
\begin{align} \label{w31}
\int_{B_{1-\eps}} w \left(-\Delta \zeta_k \right) -2 \nabla w \cdot \nabla \zeta_k\,dx=0.
\end{align}
Moreover, by the divergence theorem, we have that
\begin{align} \label{intw}
\begin{split}
\int_{B \backslash B_{1-\eps}} w \left(-\Delta \zeta_k \right) -2 \nabla w \cdot \nabla \zeta_k \,dx&=\int_{B \backslash B_{1-\eps}} w \Delta \zeta_k\,dx\\
&=\omega_{N-1} \int_0^{\eps} w(1-\rho)\Delta \zeta_k(1-\rho)(1-\rho)^{N-1}\,d\rho\\
&=\frac {\omega_{N-1}} {k} \int_0^{\eps k} w(1- k^{-1}\rho)\Delta \zeta_k(1-k^{-1}\rho)(1-k^{-1}\rho)^{N-1}\,d\rho.
\end{split}
\end{align}
Since $kw(1- k^{-1}\rho) \leq C \rho$ by \cite[Lemma 2.1]{BM},  $\mbox{supp} \,\chi'  \subset [1, 2]$ and $\mbox{supp} \, \chi'' \subset [1, 2]$, by \eqref{defsd}, then
$$
\frac 1 k \left|w(1- k^{-1}\rho) \Delta \zeta_k(1-k^{-1}\rho)\right| \leq \frac{C\rho}{1+\rho^3}, \quad \forall \,\, 0 < \rho<k \eps,
$$
In addition, 
by \eqref{defsd}, we have that
$$
\lim_{k \to \infty }\frac{w(1- k^{-1}\rho)\Delta \zeta_k(1-k^{-1}\rho)(1-k^{-1}\rho)^{N-1}}{k}=\psi_{w}(1) \rho \chi''(\rho).
$$
Invoking the dominated convergence theorem, we then obtain that
$$
\lim_{k \to \infty} \int_{B \backslash B_{1-\eps}} w \left(-\Delta \zeta_k \right) -2 \nabla w \cdot \nabla \zeta_k \,dx=\omega_{N-1} \psi_{w}(1) \int_0^{+\infty} \rho \chi''(\rho) \,d\rho.
$$
Combining this with \eqref{w31} and the fact that $\mbox{supp} \, \chi'' \subset [1, 2]$, we then know that
\begin{align} \label{w04}
\lim_{k \to \infty} \int_{B} w \left(-\Delta \zeta_k \right) -2 \nabla w \cdot \nabla \zeta_k \,dx=\omega_{N-1} \psi_{w}(1) \int_1^{2} \rho \chi''(\rho) \,d\rho=\omega_{N-1} \psi_{w}(1).
\end{align}
At this point, going backing to \eqref{w0} and using \eqref{w01}, \eqref{w02}, \eqref{w05}, \eqref{w03} and \eqref{w04}, we arrive at
\begin{align} \label{w06}
\int_{B} -\lambda w + (p-1) u^{p-2} w \,dx+\int_{\R^N \backslash \overline{B}} (-\Delta)^s w \,dx+\omega_{N-1} \psi_{w}(1)=0.
\end{align}
Observe that
$$
\int_{B}-\Delta w+ (-\Delta)^s w-(p-1)u^{p-2}w\,dx = \int_{B}-\lambda w \,dx.
$$
Taking into account \eqref{w06}, we then have that
$$
\int_{B} -\Delta w \,dx + \int_{\R^N}  (-\Delta)^s w \,dx+\omega_{N-1} \psi_{w}(1)=0.
$$
Observe that
$$
\int_{\R^N}  (-\Delta)^s w \,dx=0.
$$
As a consequence of the divergence theorem, we then derive that
$$
\int_{\partial B} \psi_w(1) \,dS+\omega_{N-1} \psi_{w}(1)=2\omega_{N-1} \psi_{w}(1)=0.
$$
We then reach a contradiction from the fact that $\psi_w(1)<0$ by Lemma \ref{hopf}. It in turn indicates that $-\lambda$ is not an eigenvalue and the proof is complete.
\end{proof}

\begin{thm} \label{thm11}
Let $u \in H^1_0(B)$ be a ground state to \eqref{equ}. Then it is non-degenerate in $H^1_{0,r}(B)$.
\end{thm}
\begin{proof}
Since $u \in H^1_0(B)$ is a solution to \eqref{equ}, then
$$
\int_{B} |\nabla u|^2 \,dx +[w]_s^2  + \lambda \int_{B} u^2 \,dx  - (p-1) \int_{B} u^p \,dx =-(p-2) \int_{B} u^p \,dx<0.
$$
It follows that
\begin{align} \label{n1}
\int_{B} |\nabla u|^2 \,dx + [w]_s^2  - (p-1) \int_{B} u^p \,dx <- \lambda \int_{B} u^2 \,dx.
\end{align}
Let $\sigma_1$ be the first eigenvalue to \eqref{equ1}. Therefore, by \eqref{n1}, we know that $\sigma_1<\lambda$.
Since $u \in H^1_0(B)$ is a ground state to \eqref{equ}, then
$$
\int_{B} |\nabla w|^2 \,dx + [w]_s^2 + \lambda \int_{B} w^2 \,dx  - (p-1) \int_{B} u^{p-2} w^2 \,dx \geq 0, \quad \forall \,\,  w \in H^1_0(B), \int_{B} u^{p-1} w \,dx=0.
$$
It then shows that $\sigma_2\geq -\lambda$. Further, by Lemma \ref{sign}, we have that $\sigma_2> -\lambda$, namely it holds that $\sigma_1<-\lambda<\sigma_2$.
Consequently, we get that $-\lambda$ is not an eigenvalue to \eqref{equ1} and $u$ is non-degenerate in $H^1_{0,r}(B)$. Thus the proof is complete.
\end{proof}

In what follows, we are going to investigate the non-degeneracy of ground states in the space of non-radially symmetric functions $H_{0,nr}^1(B):=H_{0}^1(B) \backslash H_{0,r}^1(B)$.

\subsection{Non-degeneracy of ground states in $H_{0,nr}^1(B)$} 
Let us introduce the following eigenvalue problem in $H^1_{0,nr}(B)$,
\begin{align} \label{eignr}
-\Delta w + (-\Delta)^s w + \lambda w=\Lambda u^{p-1} w \quad \mbox{in} \,\,\, B, \quad w=0 \quad \mbox{in} \,\,\, \R^N \backslash {B}.
\end{align}
Since $u>0$ in $B$ is a ground state to \eqref{equ}, then $\Lambda_1=1$, where $\Lambda_1$ denotes the first eigenvalue to \eqref{eignr}.

\begin{thm} \label{thm12}
Let $u \in H^1_0(B)$ be a ground state to \eqref{equ}. Then it is non-degenerate in $H^1_{0,nr}(B)$.
\end{thm}
\begin{proof}
To prove the desired conclusion, we only need to assert that $\Lambda_2>p-1$, where $\Lambda_2$ denotes the seond eigenvalue to \eqref{eignr}. Assume by contradiction that $\Lambda_2 \leq p-1$. Let $w \in H^1_{0, nr}(B)$ be an eigenfunction to \eqref{eignr} corresponding to $\Lambda_2$. Let $\nu \in \mathbb{S}^{N-1}$ and $\sigma_{\nu}$ be the reflection with respect to $T_{\nu}$, where
$$
T_{\nu}:=\left\{ x \in \R^N : x \cdot \nu=0 \right\}.
$$
Since $w$ is non-radially symmetric, then there exists $\nu \in \mathbb{S}^{N-1}$ such that $w \neq w \circ \sigma_{\nu}$. For simplicity, we shall assume that $\nu=e_1$ and $T_{e_1}=\left\{ x \in \R^N : x_1=0 \right\}$. Let
$$
\Sigma^+:=\left\{x \in \R^N : x_1 > 0\right\}, \quad \Sigma^-:=\left\{x \in \R^N : x_1 < 0\right\}.
$$
Define
$$
\widetilde{w}:= \frac{w-w \circ \sigma_{e_1}}{2} \in H^1_0(B) \backslash \{0\}.
$$
Since $u$ is radially symmetric, then $\widetilde{w}$ solves the equation
\begin{align} \label{eignr1}
-\Delta \widetilde{w} + (-\Delta)^s \widetilde{w} + \lambda \widetilde{w}=\Lambda_2 u^{p-2} \widetilde{w} \quad \mbox{in} \,\,\, B, \quad \widetilde{w}=0 \quad \mbox{in} \,\,\, \R^N \backslash {B}.
\end{align}
In addition, we know that $\widetilde{w}$ is antisymmetric with respect to the hyperplane $T_{e_1}$. 
Define $v=-\partial_{x_1} u$. We conclude hat $v$ satisfies the equation
\begin{align} \label{equv}
-\Delta v + (-\Delta)^s v + \lambda v=(p-1) u^{p-2} v \quad \mbox{in} \,\,\, B, \quad v=0 \quad \mbox{in} \,\,\,\R^N \backslash {B}.
\end{align}
Since $u$ is radially symmetric and strictly decreasing in the radial direction, then $v$ is also antisymmetric with respect to the hyperplane $T_{e_1}$ and $v>0$ in $\{x \in B: x_1>0\}$. 

We claim that $\partial_{x_1} v>0$ on $\{x \in B: x_1=0\}$. Let $x_0 \in \{x \in B: x_1=0\}$ and $\rho>0$ be small enough such that $B_{\rho}(x_0) \Subset B$. Define $B_*:=B_{\rho}(4\rho, x_0')$ and $B_*':=T_{e_1}(B)$, where $x_0=(0, x_0')$. Define $\theta:=\inf_{x \in B_*} v(x)>0$. Let $g \in H^1_0(B)$ be the positive radially symmetric solution to the equation
$$
-\Delta g + (-\Delta)^s g =1 \quad \mbox{in} \,\,\, B, \quad g=0 \quad \mbox{in} \,\,\, \R^N \backslash {B}.
$$
Let $\phi \in C^{\infty}_0(B_*)$, $0 \leq \phi \leq 1$ and $\phi=1$ in $\widetilde{B} \Subset B_*$ with $|\widetilde{B}|>0$. Define
$$
h(x):=\kappa x_1 g(x) +\theta \phi(x) -\theta \phi (\sigma_{e_1}(x)).
$$
Taking $\kappa>0$ be small enough and using the spirit of the proof of \cite[Lemma 3.3]{BJ}, we then have that
$$
-\Delta h + (-\Delta)^s h+(\lambda -(p-1)u^{p-2})h \leq 0 \quad \mbox{in} \,\,\, B_{\rho}^+(x_0),
$$
where $B_{\rho}^+(x_0):= B_{\rho}(x_0)\cap \Sigma^+$. It is simple to check that $v-h$ is antisymmetric with respect to the hyperplane $T_{e_1}$ and $v-\sigma h \geq 0$ in $ \Sigma^+ \backslash B_{\rho}^+(x_0)$ for some small constant $\sigma>0$. Furthermore, it holds that
$$
-\Delta (v-\sigma h)+ (-\Delta)^s (v-\sigma h)+(\lambda -(p-1)u^{p-2})(v-\sigma h) \geq 0 \quad \mbox{in} \,\,\, B_{\rho}^+(x_0).
$$
Taking into account \cite[Lemma 3.1]{BMS}, we then get that $v \geq h=\sigma \kappa x_1 g$ in $B_{\rho}^+(x_0)$. It then indicates that $\partial_{x_1} v>0$ on $\{x \in B: x_1=0\}$. As a consequence, the function $x \mapsto \frac{v(x)}{x_1}$ can be extended to a positive smooth function in $B$. 

Define $\widetilde{w}_k:=\widetilde{w} \zeta_k \in C_0(B)$, where $\chi \in C^{\infty}_0(\R)$ be such that $0 \leq \chi \leq 1$ in $\R$, $\chi=1$ in $[-1, 1]$, $\chi=0$ in $ \R \backslash (-2, 2)$ and
$$
\zeta_k(x):=1-\chi(k(1-|x|)), \quad x \in \R^N.
$$
Testing \eqref{equv} by $\frac{\widetilde{w}_k^2}{v}$ and using the divergence theorem, we then obtain that
\begin{align} \label{v12}
\begin{split}
C(N, s) \int_{\R^N} \int_{\R^N} \frac{\widetilde{w}_k^2(x)}{v(x)} \frac{v(x)-v(y)}{|x-y|^{N+2s}} \,dxdy&=\int_{B}\frac{\widetilde{w}_k^2}{v} (-\Delta)^s v \,dx \\
&=\int_{B} \frac{\widetilde{w}_k^2}{v} \Delta v \,dx + \int_{B} \left((p-1)u^{p-1} -\lambda\right) \widetilde{w}_k^2 \,dx \\
&=-\int_{B} \nabla \left(\frac{\widetilde{w}_k^2}{v}\right) \cdot \nabla v \,dx + \int_{B} \left((p-1)u^{p-1} -\lambda\right) \widetilde{w}_k^2 \,dx.
\end{split}
\end{align}
Observe that
$$
\nabla \left(\frac{\widetilde{w}_k^2}{v}\right) \cdot \nabla v=2\frac{\widetilde{w}_k}{v} \nabla \widetilde{w}_k \cdot \nabla v-\frac{\widetilde{w}_k^2}{v^2} |\nabla v|^2 \leq |\nabla \widetilde{w}_k|^2,
$$
\begin{align*}
&\int_{\R^N} \int_{\R^N} \frac{\widetilde{w}_k^2(x)}{v(x)} \frac{v(x)-v(y)}{|x-y|^{N+2s}} \,dxdy
=\frac 12 \int_{\R^N} \int_{\R^N} \left(\frac{\widetilde{w}_k^2(x)}{v(x)}-\frac{\widetilde{w}_k^2(y)}{v(y)}\right) \frac{v(x)-v(y)}{|x-y|^{N+2s}} \,dxdy \\
& = \frac 12 \int_{\R^N} \int_{\R^N}  \frac{\left(\widetilde{w}_k(x)-\widetilde{w}_k(y)\right)^2}{|x-y|^{N+2s}} \,dxdy -\frac 12 \int_{\R^N} \int_{\R^N}  \frac{v(x)v(y)}{|x-y|^{N+2s}}\left(\frac{\widetilde{w}_k(x)}{v(x)}-\frac{\widetilde{w}_k(y)}{v(y)}\right)^2 \,dxdy,
\end{align*}
where we used the fact that
$$
\left(\frac{\widetilde{w}_k^2(x)}{v(x)}-\frac{\widetilde{w}_k^2(y)}{v(y)}\right) \left(v(x)-v(y)\right) = \left(\widetilde{w}_k(x)-\widetilde{w}_k(y)\right)^2-v(x)v(y) \left(\frac{\widetilde{w}_k(x)}{v(x)}-\frac{\widetilde{w}_k(y)}{v(y)}\right)^2, \quad \forall \,\, x, y \in \R^N.
$$
Going back to \eqref{v12}, we then have that
\begin{align} \label{v22} 
\begin{split}
\int_{B} \left((p-1)u^{p-1} -\lambda\right) \widetilde{w}_k^2 \,dx 
&\leq  \int_{B} |\nabla \widetilde{w}_k|^2\,dx + \frac {C(N, s)}{2} \int_{\R^N} \int_{\R^N}  \frac{\left(\widetilde{w}_k(x)-\widetilde{w}_k(y)\right)^2}{|x-y|^{N+2s}} \,dxdy\\ 
& \quad -\frac {C(N,s)}{2} \int_{\R^N} \int_{\R^N}  \frac{v(x)v(y)}{|x-y|^{N+2s}}\left(\frac{\widetilde{w}_k(x)}{v(x)}-\frac{\widetilde{w}_k(y)}{v(y)}\right)^2 \,dxdy \\ 
&= [\widetilde{w}_k]_s^2 -\frac {C(N,s)}{2} \int_{\R^N} \int_{\R^N}  \frac{v(x)v(y)}{|x-y|^{N+2s}}\left(\frac{\widetilde{w}_k(x)}{v(x)}-\frac{\widetilde{w}_k(y)}{v(y)}\right)^2 \,dxdy \\
&  \quad +\int_{B} |\nabla \widetilde{w}_k|^2\,dx.
\end{split}
\end{align}
It is simple to compute that
\begin{align} \label{v23}
\int_{B} |\nabla \widetilde{w}_k|^2 \,dx= \int_{B} |\nabla \widetilde{w}|^2 \zeta_k^2 \,dx+2\int_{B} \widetilde{w} \zeta_k \nabla \widetilde{w} \cdot \nabla \zeta_k \,dx + \int_{B} |\nabla \zeta_k|^2 \widetilde{w}^2 \,dx.
\end{align}
Since $\lim_{k \to \infty} \zeta_k(x)=1$ for $x \in B$, by the dominated convergence theorem, then
$$
\lim_{k \to \infty}\int_{B} |\nabla \widetilde{w}|^2 \zeta_k^2 \,dx=\lim_{k \to \infty}\int_{B} |\nabla \widetilde{w}|^2\,dx.
$$
Reasoning as the proofs of \eqref{w31} and \eqref{intw} and applying \eqref{deffd}, \cite[Lemma 2.1]{BM} and \cite[Theorem 1.3]{SVWZ} along with the dominated convergence theorem, we are able to conclude that
$$
\lim_{k \to \infty} \int_{B} \widetilde{w} \zeta_k \nabla \widetilde{w} \cdot \nabla \zeta_k \,dx=0, \quad \lim_{k \to \infty} \int_{B} |\nabla \zeta_k|^2 \widetilde{w}^2 \,dx=0.
$$
It then follows from \eqref{v23} that
$$
\lim_{k \to \infty}\int_{B} |\nabla \widetilde{w}_k|^2 \,dx=\int_{B} |\nabla \widetilde{w}|^2 \,dx.
$$
Moreover, invoking the dominated convergence theorem, we know that
$$
\lim_{k \to \infty}\int_{B} \left((p-1)u^{p-1} -\lambda\right) \widetilde{w}_k^2 \,dx=\int_{B} \left((p-1)u^{p-1} -\lambda\right) \widetilde{w}^2 \,dx.
$$
Taking advantage of \cite[Lemma 2.2]{DFW}, we have that $\lim_{k \to \infty} [\widetilde{w}_k]_s^2=[\widetilde{w}]_s^2$. 
Observe that
\begin{align*}
\int_{\R^N} \int_{\R^N} \frac{v(x)v(y)}{|x-y|^{N+2s}}\left(\frac{\widetilde{w}_k(x)}{v(x)}-\frac{\widetilde{w}_k(y)}{v(y)}\right)^2 \,dxdy
&=2\int_{\Sigma_{x_1}^+} \int_{\Sigma_{y_1}^+} H_v^{e_1}(x,y )\left(\frac{\widetilde{w}_k(x)}{v(x)}-\frac{\widetilde{w}_k(y)}{v(y)}\right)^2 \,dxdy,
\end{align*}
where 
$$
\Sigma_{x_1}^+:=\left\{x \in \R^N : x_1>0\right\}, \quad \Sigma_{y_1}^+:=\left\{x \in \R^N : y_1>0\right\},
$$
$$
H_v^{e_1}(x, y):= v(x)v(y) \left(\frac{1}{|x-y|^{N+2s}}- \frac{1}{|x-\sigma_{e_1}(y)|^{N+2s}}\right), \quad x \in \Sigma_{x_1}^+, y \in \Sigma_{y_1}^+.
$$
If $x \in \Sigma_{x_1}^+$ and $y \in \Sigma_{y_1}^+$, then $|x-y| < |x-\sigma_{e_1}(y)|$. Therefore, we get that $H_v^{e_1} \geq 0$ in $\Sigma_{x_1}^+ \times \Sigma_{y_1}^+$.
At this point, make using of \eqref{v22} and Fatou's lemma, we then arrive at
$$
C(N,s)\int_{\Sigma_{x_1}^+} \int_{\Sigma_{y_1}^+} H_v^{e_1}(x,y )\left(\frac{\widetilde{w}(x)}{v(x)}-\frac{\widetilde{w}(y)}{v(y)}\right)^2 \,dxdy \leq \int_{B} |\nabla \widetilde{w}|^2\,dx + [\widetilde{w}]_s^2 - \int_{B} \left((p-1)u^{p-1} -\lambda\right) \widetilde{w}^2 \,dx.
$$
Since $\widetilde{w} \in H^1_0(B)$ solves \eqref{eignr1} and $\Lambda \leq p-1$ by the assumption, then
\begin{align*}
\int_{B} |\nabla \widetilde{w}|^2\,dx + [\widetilde{w}]_s^2 &=\int_{B} \left(\Lambda_2 u^{p-1} -\lambda\right) \widetilde{w}^2 \,dx \leq \int_{B} \left((p-1)u^{p-1} -\lambda\right) \widetilde{w}^2 \,dx \\
&\leq \int_{B} |\nabla \widetilde{w}|^2\,dx + [\widetilde{w}]_s^2 -C(N,s)\int_{\Sigma_{x_1}^+} \int_{\Sigma_{y_1}^+} H_v^{e_1}(x,y )\left(\frac{\widetilde{w}(x)}{v(x)}-\frac{\widetilde{w}(y)}{v(y)}\right)^2 \,dxdy.\\
\end{align*}
It then leads to
$$
\int_{\Sigma_{x_1}^+} \int_{\Sigma_{y_1}^+} H_v^{e_1}(x, y)\left(\frac{\widetilde{w}(x)}{v(x)}-\frac{\widetilde{w}(y)}{v(y)}\right)^2 \,dxdy \leq 0.
$$
At this point, we have that
$$
\frac{\widetilde{w}(x)}{v(x)}=\frac{\widetilde{w}(y)}{v(y)}, \quad \forall \,\, x \in \Sigma_{x_1}^+ \cap B, y \in \Sigma_{y_1}^+ \cap B.
$$
This shows that there exists $k \in \R$ such that $\widetilde{w}=kv$ in $\Sigma_{x_1}^+ \cap B$. Since $\widetilde{w}(x)=0$
and $v(x)>0$ as $|x| \to 1$ in $x \in \Sigma_{x_1}^+$, then $k=0$. This shows that $\widetilde{w}=0$. We then reach a contradiction. Therefore, we conclude that $\Lambda_2>p-1$ and $u$ is non-degenerate in $H^1_{nr}(B)$. Thus the proof is complete.
\end{proof}

\begin{proof}[Proof of Theorem \ref{thm2}]
Combining Theorems \ref{thm11} and \ref{thm12}, we then have the desired conclusion. This completes the proof.
\end{proof}

\section{Uniqueness of ground states} \label{uniqueness}

Relying on the non-degeneracy of ground states obtained in the previous section, we are now in a position to establish the uniqueness of ground states. For this aim, we first prove the uniqueness of the solutions for $p$ close to $2$, whose proof is inspired by the one of \cite[Theorem 1.6]{DIS}.


\begin{proof} [Proof of Theorem \ref{thm3}]
Let $\{p_n\} \subset (2, 2^*)$ be such that $p_n \to 2$ as $n \to \infty$. Let $u_n \in H^1_0(B)$ be a ground state to \eqref{equ} with $p=p_n$. Define $M_n:=\|u_n\|_{\infty}$. First we shall prove that $\{M_n^{p_n-2}\}$ is bounded in $\R$ by the blow-up argument introduced in \cite{GS}. Let us argue by contradiction. Therefore, we may assume that $M_n^{p_n-2} \to +\infty$ as $n \to \infty$. Let $x_n \in B$ be such that $M_n=u_n(x_n)$. Define
$$
v_n:=\frac{1}{M_n}u_n(\mu_n \cdot +x_n), \quad \mu_n:=M_n^{-\frac{p_n-2}{2}}.
$$
It is obvious that $0<v_n \leq 1$ and $v_n(0)=1$. Moreover, we know that
\begin{align} \label{equ311}
-\Delta v_n + \mu_n^{2-2s}(-\Delta)^s v_n =f_n,  \quad \mbox{in} \,\,\, B_n, \quad v_n=0 \quad \mbox{in} \,\,\,\R^N \backslash B_n,
\end{align}
where
$$
f_n:=v_n^{p_n-1}-\frac{\lambda}{M_n^{p_n-2}}  v_n,
$$
$$
B_n:=\left\{ y \in \R^N : \mu_n y + x_n \in B\right\}.
$$

If $\mbox{dist}(x_n, \partial B) \mu_n^{-1} \to + \infty$ as $n \to \infty$, then $B_n \to \R^N$ as $n \to \infty$. Let $R>0$. It then holds that $\overline{B_{4R}} \subset B_n$ for $n \in \N^+$ large enough. Using the arguments as the proof of \cite[Theorem 1.3]{SVWZ}, we then get that, for any $0<\alpha< \min\{1, 2(1-s)\}$,
$$
\|v_n\|_{C^{1, \alpha}(\overline{B_{2R}})} \leq C_1,
$$
where $C_1=C_1(s, \alpha, N, R)>0$. 
Further, applying the arguments as the proof of \cite[Theorem 1.5]{SVWZ}, we then obtain that 
$$
\|v_n\|_{C^{2, \alpha}(\overline{B_{R}})} \leq C_2,
$$
where $C_2=C_2(s, \alpha, N, R)>0$. Taking into account Arzel\`a-Ascoli's theorem, 
we then know that there exists $v \in C^2(\R^N)$ solves the equation
$$
-\Delta v=v \quad \mbox{in}\,\, \R^N.
$$
Since $v_n>0$, then $v \geq 0$. By the strong maximum principle, we have that $v>0$ in $\R^N$. However, by Liouville's theorem, we get that $v=0$ in $\R^N$. This is impossible.

If $\mbox{dist}(x_n, \partial B) \mu_n^{-1} \to d \geq 0$ as $n \to \infty$ for some $d < +\infty$, then $x_n \to x_0 \in \partial B$ as $n \to \infty$. Without restriction, we may assume that 
$x_0=(0, 0, \cdots, -1)$. Define
$$
w_n:=\frac{1}{M_n} u_n(\mu_n \cdot + \xi_n), \quad \mu_n:=M_n^{-\frac{p_n-2}{2}}.
$$
where $\xi_n \in \partial B$ is the projection of $x_n$ on $\partial B$. Define
$$
D_n:=\left\{y \in \R^N : \mu_n y+ \xi_n \in B \right\}.
$$
It is straightforward to see that $0 \in \partial D_n$ and $D_n \to \R^N_+$ as $n \to \infty$. In addition, we find that
$$
-\Delta w_n + \mu_n^{2-2s}(-\Delta)^s w_n =g_n,  \quad\mbox{in} \,\,\, D_n, \quad w_n=0 \quad \mbox{in} \,\,\, \R^N \backslash D_n,
$$
where
$$
g_n:=w_n^{p_n-1}-\frac{\lambda}{M_n^{p_n-2}}  w_n.
$$
Define
$$
y_n:=\frac{x_n-\xi_n}{\mu_n} \in D_n
$$
be such that $w_n(y_n)=1$. Since $\xi_n \in \partial B$ is the projection of $x_n$ on $\partial B$, then $|y_n|=\mbox{dist}(x_n, \partial B) \mu_n^{-1}$. Following closely the proof of \cite[Lemma 2.1]{BMS}, we can conclude that 
$$
w_n(y) \leq C \mbox{dist}(y, \R^N \backslash D_n), \quad \forall \,\, y \in D_n,
$$ 
where $C=C(N, s)>0$. It then gives that there exists $\delta_0>0$ such that $w_n(y) <\frac 12$ for any $y \in D_n$ and $\mbox{dist}(y, \partial D_n)<\delta_0$.
Let us now prove that $d>0$. If not, then $d=0$, namely $|y_n|=\mbox{dist}(x_n, \partial B) \mu_n^{-1} \to 0$ as $n \to \infty$. Since $\mbox{dist}(y_n, \partial D_n) \leq |y_n|$ by the fact that $0 \in \partial D$, then $\mbox{dist}(y_n, \partial D_n)<\delta_0$ for $n \in \N^+$ large enough. Therefore, we obtain that $w_n(y_n) < \frac 12$. This is impossible. It then shows that $d>0$. By a similar way as previously, we then know that there exists $w \in C^2(\R^+_N)$ solving the equation
$$
-\Delta v=v \quad \mbox{in}\,\,\, \R^N_+.
$$
In addition, since $d>0$, then there exists $y_0 \in \R^+_N$ such that $y_n \to y_0$ as $n \to \infty$ and $w(y_0)=1$, namely $w \neq 0$ in $\R^+_N$. However, by Liouville's theorem, it holds that $w=0$ in $\R^+_N$. We then get a contradiction. It then shows that $\{M_n^{p_n-2}\}$ is bounded in $\R$.

Without restriction, we may assume that $M_n^{p_n-2} \to \mu$ as $n \to \infty$ for some $\mu>0$. Define $z_n:=\frac{u_n}{M_n}$. Therefore, we find that $0<z_n \leq 1$ and $z_n$ satisfies the equation
$$
-\Delta z_n + (-\Delta)^s z_n =h_n  \quad\mbox{in} \,\,\, B, \quad z_n=0 \quad \mbox{in} \,\,\, \R^N \backslash B,
$$
where
$$
h_n:=-\lambda z_n +M_n^{p_n-2}z_n^{p_n-1}.
$$
In view of \cite[Theorems 1.4 $\&$ 1.5]{SVWZ} and Arzel\`a-Ascoli's theorem, we similarly have that
there exists $z \in C^{2}(B)$ and $z \geq 0$ satisfying the equation
$$
-\Delta z + (-\Delta)^s z =\left(-\lambda + \mu \right) z \quad\mbox{in} \,\,\, B, \quad z=0 \quad \mbox{in} \,\,\, \R^N \backslash B.
$$
Furthermore, by the maximum principle, we derive that $z>0$ in $B$. This indicates that $\mu=\lambda_{1,s} + \lambda$ and $z=\varphi_{1,s}$ where $\varphi_{1,s}$ is the eigenfunction to the following eigenvalue problem corresponding to the first eigenvalue $\lambda_{1,s}$ in $H^1_0(B)$,
\begin{align} \label{equ30}
-\Delta \varphi+ (-\Delta)^s \varphi =\lambda \varphi \quad \mbox{in} \,\,\, B, \quad \varphi=0 \quad \mbox{in} \,\,\, \R^N \backslash {B}.
\end{align}

We now suppose by contradiction that \eqref{equ} admits two ground states to $u_n, v_n \in H^1_0(B)$ with $p=p_n$ and $p_n \to 2$ as $n \to \infty$. Define
$$
w_n:=\frac{u_n-v_n}{\|u_n-v_n\|_{2}}.
$$
Clearly, we see that $w_n \in H^1_0(B)$ satisfies the equation
\begin{align} \label{equ31}
-\Delta w_n + (-\Delta)^s w_n =-\lambda w_n + \alpha_n w_n  \quad\mbox{in} \,\,\,  B, \quad w_n=0 \quad \mbox{in} \,\,\, \R^N \backslash B,
\end{align}
where
$$
\alpha_n(x):=(p_n-1) \int_0^1 \left(tu_n(x) +(1-t)v_n(x)\right)^{p_n-2} \,dt, \quad x \in B.
$$
Note that
$$
\lim_{n \to \infty} (p_n-1)u_n^{p_n-2} = \lim_{n \to \infty} (p_n-1)M_n^{p_n-2} \left(\frac{u_n}{M_n}\right)^{p_n-2}=\lambda_{1,s} + \lambda, \quad \forall \,\, x \in B.
$$
Similarly, we know that $(p_n-1)v_n^{p_n-2} \to \lambda_{1,s} + \lambda$ for any $x \in B$ as $n \to \infty$. Since $\|u_n\|_{\infty} \leq C$ and $\|v_n\|_{\infty} \leq C$ for some $C>0$, by the dominated convergence theorem, then $\alpha_n(x) \to \lambda_{1,s} + \lambda$ for any $x \in B$ as $n \to \infty$. Noting that $\|\alpha_n\|_{\infty} \leq C$ 
and testing \eqref{equ31} by $w_n$, we then get that $\{w_n\}$ is bounded in $H^1_0(B)$. As a consequence, there exists $w \in H^1_0(B) \backslash \{0\}$ such that $w_n \wto w$ in $H_0^1(B)$ as $n \to \infty$ satisfying the equation
$$
-\Delta w + (-\Delta)^s w =\lambda_{s,1} w  \quad\mbox{in} \,\,\, B, \quad w=0 \quad \mbox{in} \,\,\, \R^N \backslash B.
$$
It then follows that $w=\varphi_{1,s}$. Invoking \cite[Theorem 1.1]{SVWZ}, we find that $\|w_n\|_{\infty} \leq C$. Using \cite[Theorem 1.3]{SVWZ} and Arzel\`a-Ascoli’s theorem, we then derive that $w_n  \to \varphi_{1,s}$ in $C(\overline{B})$ as $n \to \infty$. Therefore, adapting  \cite[Theorem 1.2]{BMS} and Arzel\`a-Ascoli’s theorem, 
we obtain that
\begin{align} \label{hopfwn}
\frac{w_n}{1-|x|} \to \frac{\varphi_{1,s}}{1-|x|} \quad \mbox{in} \,\, C(\overline{B}) \,\, \mbox{as} \,\, n \to \infty.
\end{align}
If $w_n \geq 0$ in $B$, then $u_n \geq v_n$ and $u_n^{p_n-2} \geq v_n^{p_n-2}$ in $B$. Observe that
$$
0=\int_B \left(-\Delta u_n + (-\Delta)^s u_n\right) v_n \,dx -\int_B \left(-\Delta v_n + (-\Delta)^s v_n\right) u_n \,dx=\int_{B}u_n(u_n^{p_n-2} -v_n^{p_n-2}) v_n \,dx.
$$
It then gives that $u_n=v_n$ in $B$. This is impossible by the assumption. Accordingly, we have that $w_n$ has to change sign in $B$. Define
$$
w_n(x_n):=\min_{x \in B} w_n(x).
$$
Therefore, we know that $w_n(x_n)<0$ and $x_n \to x_0 \in \partial B$ as $n \to \infty$, because $w_n  \to \varphi_{1,s}$ in $C(\overline{B})$ as $n \to \infty$. As a result, applying \eqref{hopfwn}, we know that
$$
\lim_{n \to \infty} \frac{\varphi_{1,s}(x_n)}{1-|x_n|}=\lim_{n \to \infty} \frac{w_n(x_n)}{1-|x_n|} \leq 0.  
$$
However, using Hopf's lemma (see \cite[Theorem 2.2]{BM}), we have that
$$
\lim_{n \to \infty} \frac{\varphi_{1,s}(x_n)}{1-|x_n|}>0.
$$
We then reach a contradiction. This completes the proof.
\end{proof}

\begin{proof}[Proof of Theorem \ref{thm4}]
We shall follow the proof of \cite[Theorem 1.5]{FW1} to demonstrate the desired conclusion. Let $2<p_* \leq 2^*$ be the large number such that \eqref{equ} has a unique ground state $u_p \in H^1_0(B)$ for any $2<p<p_*$. Here the existence of $p_*$ is guaranteed by Theorem \ref{thm2}. If $p_*=2^*$, then the proof is completed. We now assume by contradiction that $2<p_*<2^*$. Define
$$ 
\mathcal{C}_0^{\beta}:=\left\{ u \in C^{0,\beta}(\R^N) : -\Delta u+ (-\Delta)^s u \in C^{0,\beta}(\overline{B}), \,\,\,  u=0 \,\,\mbox{in}\,\, \R^N \backslash B\right\}
$$
equipped with the norm
$$
\|u\|_{\mathcal{C}^{\beta}_0}:=\|u\|_{C^{0,\beta}(\R^N)} + \|\Delta u\|_{C^{0,\beta}(\overline{B})} + \|(-\Delta)^s u\|_{C^{0,\beta}(\overline{B})},
$$
where $0<\beta< \min \left\{(\widetilde{p}-2)\alpha, \alpha \right\}$, $2<\widetilde{p}<p_*$ is constant and $0<\alpha<1$ is a constant corresponding to H\"older's regularity exponential of solutions to the equation
$$
-\Delta u+ (-\Delta)^s u =f \quad \mbox{in} \,\,\, B, \quad  u=0 \quad \mbox{in} \,\,\, \R^N \backslash B.
$$
Regarding the regularity of solutions, we refer the readers to \cite[Theorem 2.7]{BDVV} or \cite[Theorem 1.3]{SVWZ}. 
Define $F: (2, +\infty) \times \mathcal{C}_0^{\beta} \to C^{0,\beta}(\overline{B})$ by
$$
F(p, u):=-\Delta u + (-\Delta)^s u+ \lambda u -|u|^{p-1}.
$$
Observe that $F(p_*, u_{p_*})=0$ and
$$
\partial_u F(p_*, u_{p_*})=-\Delta + (-\Delta)^s + \lambda -(p_*-1)u_{p_*}^{p_*-2}.
$$
In addition, we know that $Ker[\partial_u F(p_*, u_{p_*})]=0$ by Theorem \ref{thm2}. It is not hard to check that $-\Delta + (-\Delta)^s + \lambda : \mathcal{C}_0^{\beta} \to C^{0, \beta}(\overline{B})$ is a Fredholm map of index zero. Observe that $u_{p_*} \in C^{1, \alpha}(\overline{B})$ for some $0<\alpha<1$ by \cite[Theorem 2.7]{BDVV} or \cite[Theorem 1.3]{SVWZ}. In particular, it holds that $u_{p_*} \in C^{0, \alpha}(\overline{B})$. Therefore, we get that $u_{p_*}^{p_*-2} \in C^{0, \beta}(\overline{B})$, because of $0<\beta< \min \left\{(\widetilde{p}-2)\alpha, \alpha \right\}$. Invoking Arzel\`a-Ascoli's theorem, we then know that
the map $v \mapsto u_{p_*}^{p_*-2}v : \mathcal{C}_0^{\beta} \to C^{0, \beta}(\overline{B})$ is compact. Therefore, we can conclude that $\partial_u F(p_*, u_{p_*}) : \mathcal{C}_0^{\beta} \to C^{0, \beta}(\overline{B})$ is an isomorphism. Using the implicit function theorem, we then know that there exists $\delta>0$ such that, for any $p \in (p_*-\delta, p_*+\delta)$, there exists a unique $u_p \in B(u_{p_*}, \delta)$ such that $F(p, u_p)=0$, where
$$
 B(u, \delta):=\left\{v \in \mathcal{C}^{\beta}_0 : \|v-u\|_{\mathcal{C}^{\beta}_0} < \delta \right\}, \quad u \in \mathcal{C}^{\beta}_0.
$$
Let us suppose that there exists $\widetilde{u}_{p_*} \in \mathcal{C}_0^{\beta} \backslash \{0\}$ such that $F(p_*, \widetilde{u}_{p_*})=0$. Similarly, we have that there exists $\delta>0$ by taking smaller if necessary such that, for any $p \in (p_*-\delta, p_*+\delta)$, there exists a unique $\widetilde{u}_p \in B(\widetilde{u}_{p_*}, \delta)$ such that $F(p, \widetilde{u}_p)=0$. By the continuity, taking $\delta>0$ much smaller if necessary, we get that $u_p(0)>0$ and $\widetilde{u}_p(0)>0$. Therefore, by the strong maximum principle, we obtain that $u_p>0$ and $\widetilde{u}_p>0$ in $B$. In view of the definition of $p_*$, then $u_p=\widetilde{u}_p$ for any $p_*-\delta<p<p_*$. Taking $\{p_n\} \subset (p_*-\delta, p_*)$ be such that $p_n \to p_*$ as $n \to \infty$, we then derive that $u_{p_*}=\widetilde{u}_{p_*}$. It then shows that $u_{p_*}$ is the unique solution of $F(p_*, u)=0$.

Next, we shall prove that there exists $0<\eps_*<\delta$ such that, for any $p_*<p<p_*+\eps_*$, there exists a unique solution to $F(p, u)=0$. If not, then there exists a sequence $\{p_n\} \subset (p_*, p_*+\delta)$ with $p_n \to p_*$ as $n \to \infty$ such that there exist $u_{p_n}, \widetilde{u}_{p_n} \in \mathcal{C}_0^{\beta} \backslash \{0\}$ satisfying $u_{p_n} \neq \widetilde{u}_{p_n}$ and $F(p_n, u_{p_n})=F(p_n, \widetilde{u}_{p_n})=0$.
Define
$$
z_n:=\frac{u_{p_n}-\widetilde{u}_{p_n}}{\|u_{p_n}-\widetilde{u}_{p_n}\|_{2}}.
$$
It is simple to see that $z_n$ solves the equation
\begin{align} \label{equz}
-\Delta z_n + (-\Delta)^s z_n =-\lambda z_n + \theta_n z_n  \quad\mbox{in} \,\,\, B, \quad z_n=0 \quad \mbox{in} \,\,\, \R^N \backslash B,
\end{align}
where
$$
\theta_n(x):=(p_n-1) \int_0^1 \left(t u_{p_n}(x) +(1-t)\widetilde{u}_{p_n}(x)\right)^{p_n-2} \,dt, \quad x \in B.
$$
Proceeding the blow-up argument as the proof of Theorem \ref{thm2}, we are able to conclude that $\|u_{p_n}\|_{\infty} \leq C$ and $\|\widetilde{u}_{p_n}\|_{\infty}\leq C$.  It then follows that $\|\theta_n\|_{\infty} \leq C$.
Since $u_{p_n} \in H^1_0(B)$ solves the equation
\begin{align} \label{equ32}
-\Delta u_{p_n} + (-\Delta)^s u_{p_n} =-\lambda u_{p_n} + u_{p_n}^{p_n-1} \quad\mbox{in} \,\,\, B, \quad u_{p_n} =0 \quad \mbox{in} \,\,\, \R^N \backslash B.
\end{align}
then $\{u_{p_n}\}$ is bounded in $H^1_0(B)$. Moreover, according to \cite[Lemma 1.3]{SVWZ}, we have that $\|u_{p_n}\|_{C^{1,\alpha}(\overline{B})} \leq C$. As a consequence, there exists $u_{p_*} \in H^1_0(B) \backslash \{0\}$ such that $u_{p_n} \wto u_{p_*}$ in $H^1_0(B)$ and $u_{p_n} \to u_{p_*}$ in $C(\overline{B})$ as $n \to \infty$ solving the equation
\begin{align} \label{equ33}
-\Delta u_{p_*}  + (-\Delta)^s u_{p_*}  =-\lambda u_{p_*} + u_{p_*} ^{p_*-1} \quad\mbox{in} \,\,\, B, \quad u_{p_*} =0 \quad \mbox{in} \,\,\,\R^N \backslash B.
\end{align}
Since $u_{p_n} >0$ in $B$, then $u_{p^*} \geq 0$ in $B$. By the maximum principle, we then obtain that $u_{p^*} >0$ in $B$. 
In addition, since $u_{p_n} \in H^1_0(B)$ is a ground state to \eqref{equ32}, then we derive that $u_{p_*} \in H^1_0(B)$ is a ground state to \eqref{equ33}. Similarly, we get that $\widetilde{u}_{p_n} \wto u_{p_*}$ in $H^1_0(B)$ and $\widetilde{u}_{p_n} \to u_{p_*}$ in $C(\overline{B})$ as $n \to \infty$.
It then shows that $\theta_n (x) \to (p_*-1)u^{p_*-2}_{p_*}(x)$ as $n \to \infty$ for any $x \in B$. Thanks to \cite[Lemma 1.1]{SVWZ}, we have that $\|z_n\|_{\infty} \leq C$. Furthermore, we know that $\{z_n\}$ is bounded in $H_0^1(B)$. It then infers that there exits $z \in H^1_0(B) \backslash \{0\}$ satisfying the equation
$$
-\Delta z + (-\Delta)^s z =-\lambda z + (p_*-1) u_{p_*}^{p_*-2} z  \quad \mbox{in} \,\,\, B, \quad z=0 \quad \mbox{in} \,\,\, \R^N \backslash B.
$$
This is impossible by Theorem \ref{thm2}.  As a result, there exists $0<\eps_*<\delta$ such that, for any $p_*<p<p_*+\eps_*$, there exists a unique solution to $F(p, u)=0$. It clearly contradicts with the definition of $p_*$. It then indicates that $p_*=2^*$. This completes the proof.
\end{proof}

\end{document}